\documentclass[10pt]{article}

\usepackage[all]{xypic}
\usepackage{amsmath}
\usepackage{amsfonts}
\usepackage{amssymb}
\usepackage{amsthm}

\newtheorem{thm}{Theorem}[section]
\newtheorem{mthm}{Theorem}

\newtheorem{crl}[thm]{Corollary}

\newtheorem{prp}[thm]{Proposition}
\newtheorem{mprp}[mthm]{Proposition}

\newtheorem{lmm}[thm]{Lemma}

\newtheorem{exm}[thm]{Example}

\newcommand {\mb}{\mathbb}
\newcommand {\Z}{\mb Z}
\newcommand {\R}{\mb R}
\newcommand {\C}{\mb C}

\newcommand {\im}{\textrm{im}}
\newcommand {\colim}{\textrm{colim}\ }
\newcommand {\cok}{\textrm{coker}}

\newcommand {\ex}{\mathrm{excess}}

\begin{document}

\title{On the Bott periodicity, $J$-homomorphisms, and $H_*Q_0S^{-k}$}

\author{Hadi Zare}
\date{}

\maketitle

\abstract{The Curtis conjecture predicts that the only spherical classes in $H_*(Q_0S^0;\Z/2)$ are the Hopf invariant one and the Kervaire invariant one elements. We consider Sullivan's decomposition
$$Q_0S^0=J\times\cok J$$
where $J$ is the fibre of $\psi^q-1$ ($q=3$ at the prime $2$) and observe that the Curtis conjecture holds when we restrict to $J$. We then use the Bott periodicity and the $J$-homomorphism $SO\to Q_0S^0$ to define some generators in $H_*(Q_0S^0;\Z/p)$, when $p$ is any prime, and determine the type of subalgebras that they generate. For $p=2$ we determine spherical classes in $H_*(\Omega^k_0J;\Z/2)$. We determine truncated subalgebras inside $H_*(Q_0S^{-k};\Z/2)$. Applying the machinery of the Eilenberg-Moore spectral sequence we define classes that are not in the image of by the $J$-homomorphism. We shall make some partial observations on the algebraic structure of $H_*(\Omega^k_0\cok J;\Z/2)$. Finally, we shall make some comments on the problem in the case equivariant $J$-homomorphisms.}

\tableofcontents

\section{Introduction and statement of results}

Let $Q_0S^0$ be the base point component of $QS^0=\colim\Omega^i S^i$. We have the following.
\begin{mthm}
The image of the composite
$${_2\pi_*} J\rightarrowtail{_2\pi_*}Q_0S^0\stackrel{h}{\longrightarrow} H_*(Q_0S^0;\Z/2)$$
consists only of the image of the Hopf invariant one elements in positive dimensions where $h$ denotes the Hurewicz homomorphism, i.e. only $\eta,\nu,\sigma$ survive under the above composite.
\end{mthm}

We write ${_p\pi_*}$ for $p$-primary homotopy where $p$ is a prime number. We shall use $f_\#$ to denote the mapping induced in homotopy, and $f_*$ for the mapping induced in homology, where $f$ is a mapping of spaces. Given a class $\xi_i\in H_*(X;\Z/p)$, the subindex $i$ will refer to the dimension of the class. We shall use $E_{\Z/p}$ for the exterior algebra functor over $\Z/p$, $R$ for the Dyer-Lashof algebra (the algebra of Kudo-Araki operations $Q^i$ and their iterations), and $A$ for the Steenrod algebra; all mod $p$.\\

Let $Q_0S^{-k}$ be the base-point component of $\Omega^kQ_0S^0$. The real $J$-homomorphism $SO\to Q_1S^0$ is defined using the reflection through a hyperplane. After composition with a translation map $*[-1]:Q_1S^0\to Q_0S^0$, obtained by `loop sum' with a map of degree $-1$, we obtain
$$J_\R:SO\longrightarrow Q_1S^0\stackrel{*[-1]}{\longrightarrow}Q_0S^0.$$
On the other hand, the original real Bott periodicity provides a homotopy equivalence $\Z\times BO\to\Omega^8BO$. After looping once we have the homotopy equivalence $SO\to\Omega^8_0SO$ which we refer to as real Bott periodicity where $\Omega^8_0SO$ denotes the base-point component of $\Omega^8SO$. Consider the iterated loops of $J_\R$ as mappings
$$\Omega^k J_\R:\Omega^kSO\longrightarrow QS^{-k}=\Omega^kQ_0S^0.$$
The $\Z/p$-homology of the spaces $\Omega^kSO$ are completely known. Let $x^{-k}_i\in H_i(Q_0S^{-k};\Z/2)$ be the image of the algebra generator of $H_*(\Omega^k_0SO;\Z/2)$ in dimension $i$ under $(\Omega^kJ_\R)_*$ if there is such a generator, e.g. $H_*(\Omega^{8j}_0SO;\Z/2)\simeq E_{\Z/2}(s_i:i>0)$ and $x^{-8j}=(\Omega^{8j}J_\R)_*s_i$. We then have the following.

\begin{mthm}
The classes $x^{-k}_i\in H_i(Q_0S^{-k};\Z/2)$ are nontrivial if and only if:
$(1)$ If $k\not\equiv 0$ mod $8$; or $(2)$ If $k=8j$ and $i\neq 2^{\nu_j}k-1$ where $2^{\nu_j}$ is the largest power of $2$ which divides $3^{4j}-1$.\\
The action of the mod $2$ Steenrod operations on the classes $Q^Ix_i^{-k}$ is determined by the Nishida relations together with their action on $x^{-k}_i$. If $k\equiv 0,1,2,4$ then the subalgebra inside $H_*Q_0S^{-k}$ spanned by the classes $Q^Ix_i^{-k}$ is an exterior algebra. Moreover, if $k>7$ then $x_i^{-k}$ belongs to the kernel of the iterated homology suspension $\sigma_*^t:H_*Q_0S^{-k}\to H_{*+t}Q_0S^{-k+t}$ for some $t\leqslant 7$.
\end{mthm}

Next, note that any generator $\alpha\in{_2\pi_1}Q_0S^{-n+1}$ maps nontrivially under the Hurewicz homomorphism ${_2\pi_1}Q_0S^{-n+1}\to H_1(Q_0S^{-n+1};\Z/2)$, see Lemma 5.6. We determine how far any such class will survive under the homology suspension $H_1(Q_0S^{-n+1};\Z/2)\to H_{i+1}Q_0S^{-n+i+1}$ when $\alpha$ is a generator in the image of the $J$-homomorphism, i.e. we determine largest such $i$.
\begin{mthm}
The class $x^{-n}_d$ is spherical if and only if $d\leqslant 8$ and there is an element $f\in{_2\pi_d}Q_0S^{-n}$ which is in the image of $(\Omega^nJ_\R)_\#$ with $hf=x^{-n}_d\in H_dQ_0S^{-n}$. Moreover, suppose that $n=8t+k$ with $0\leqslant k\leqslant 7$ and $f=(\Omega^nJ_\R)f'$ for some $f'\in{_2\pi_d}\Omega^n_0SO$. Then under the composition
$${_2\pi_d}\Omega^n_0SO = {_2\pi_d}\Omega^{8t+k}_0SO\stackrel{\simeq}{\longrightarrow}{_2\pi_d}\Omega^k_0SO\stackrel{\simeq}{\longrightarrow}
{_2\pi_{d+k}}SO\stackrel{J_\R}{\longrightarrow}{_2\pi_{d+k}^S}$$
the class $f'$ maps to a Hopf invariant one element.
\end{mthm}

The precise value for $d$ depends on $k$, and can be determined as explained in Lemma 6.1. We shall also determine the Hurewicz image of the elements $\mu_{8i+1}$ and $\eta\mu_{8i+1}$ (see Lemma 6.2), which are not in the image of the $J$-homomorphism, and observe that `they are detected on very low dimensional spheres'.\\

To state our next result, we fix some notation. Given a generator $\gamma\in{_2\pi_k^S}\simeq{_2\pi_0}Q_0S^{-k}$ let $[\gamma]=h\gamma$ where $h:{_2\pi_0}QS^{-k}\to H_0(QS^{-k};\Z/2)$ is the Hurewicz homomorphism.
\begin{mprp}
Assume that ${_2\pi_k^S}\simeq{_2\pi_0}QS^{-k}$ has a summand which is the cyclic group of order $2^d$, and let $\gamma^k$ be the generator of this component. Then the homology algebra $H_*(Q_0S^{-k};\Z/2)$ contains a truncated polynomial subalgebra whose truncation height is at most $2^d$ and the generators of this subalgebra are of the form $Q^I[\gamma^k_J]$
where
$$
\begin{array}{ll}
[\gamma^k_J]=Q^J[\gamma^k]*[-2^{l(J)}\gamma^k]\in H_{\dim(J)}Q_0S^{-k} & \textrm{if }l(J)<d,\\

[\gamma^k_J]=Q^J[\gamma^k]\in H_{\dim(J)}Q_0S^{-k} & \textrm{if }l(J)=d.
\end{array}
$$

We assume both $I$ and $J$ are admissible, and we allow $I$ to be the empty sequence but not $J$. Moreover, the action of the Steenrod operations $Sq^r_*$ on these subalgebras is determined by the Nishida relations.
\end{mprp}

Note that $\pi_*J$ is completely known \cite[Proposition 1.5.22]{26}. Hence, we know the type of the truncated subalgebras that they give rise to.\\

Next, consider the complex $J$-homomorphism $J_\C:U\to Q_0S^0$ and its iterated loops. Recall that for primes $p>2$ we have $\im(J_\C)_\#\simeq{_p\pi_*} J$ \cite[Theorem 1.5.19]{26} where $(J_\C)_\#$ is the mapping induces in homotopy and $J$ is the fibre of $\psi^q-1:BU\to BU$ with $q$ a prime such that $q^i-1\not\equiv 0$ mod $p$ for $0\leqslant i\leqslant p-2$, and $p$ divides $q^{p-1}-1$ exactly once. Notice that there are many primes like this, and we may choose the least one for simplicity and denote with $q(p)$.\\
Let $p\geqslant 2$ be a prime number, and $k\geqslant0$ any integer and let $\epsilon(k)=1$ if $k$ is odd, and $\epsilon(k)=0$ otherwise. Let $w_{2i+1-\epsilon}^{-k}$ be the image of the generator of $H_{2i+1-\epsilon(k)}(\Omega^k_0U;\Z/p)$ under $(\Omega^k J_\C)_*$. Notice that these are primitive if $k$ is even.

\begin{mthm}
The classes $w_{2i+1-\epsilon}^{-k}\in H_{2i+1-\epsilon(k)}(Q_0S^{-k};\Z/p)$ are nontrivial if and only if
$(1)$ $p=2$ and $k\equiv 0,1,7$ mod $8$; or $(2)$ $p=2$ and $k\equiv 2,3$ mod $8$ and $i$ is even; or $(3)$ $p=2$ and $k\equiv 4$ mod $8$ and $i$ is odd; or $(4)$ $p>2$, $k=2n-1$, and $p|q(p)^{n+i}-1$; or $(5)$ $p>2$, $k=2n$, and  $p|q(p)^{n+i}-1$.\\
The action of the mod $p$ Steenrod operations on the classes $Q^Iw_{2i+1-\epsilon}^{-k}$ is determined by their action on $w_{2i+1-\epsilon}^{-k}$ together with the Nishida relations. If $k$ is even, then the classes $Q^Iw_{2i+1-\epsilon}^{-k}$ generate exterior subalgebra in $H_*(Q_0S^{-k};\Z/2)$ whereas $w_{2i+1-\epsilon}^{-k}$ generate an exterior subalgebra in $H_*(Q_0S^{-k};\Z/2)$ for $p>2$. Moreover, for $k$ even and $p>2$ the classes $Q^Iw_{2i+1-\epsilon}^{-k}$ generate a subalgebra whose truncation height is $p$ if $I$ is odd dimensional, and an exterior algebra otherwise.\\
Finally, if $k>3$ then $\sigma_*^tw_{2i+1-\epsilon}^{-k}=0$ for some $t\leqslant 3$.
\end{mthm}
The action of the mod $p$ Steenrod operations on the classes $w_{2i+1-\epsilon}^{-k}$ is determined by their action on the classes $c_{2i}\in H_{2i}(BU;\Z/p),u_{2i+1}\in H_{2i+1}(U;\Z/p)$.\\
Next, we use the machinery of the Eilenberg-Moore spectral sequence. The mapping $\Omega^{2k+1}J_\C:BU\to Q_0S^{-2k-1}$ extends to an infinite loop map $j^\C_{2k+1}:QBU\to Q_0S^{-2k-1}$. We consider
$$\Omega j^\C_{2k+1}:Q\Sigma^{-1}BU\to Q_0S^{-2k-2}.$$
The inclusion $\iota_{BU}:BU\to QBU$ gives the mapping $\Omega \iota_{BU}:U\to Q\Sigma^{-1}BU$, providing us with a factorisation of  $\Omega^{2k+2}J_\C$ as
$$\Omega^{2k+2}J_\C:U\to Q\Sigma^{-1}BU\to Q_0S^{-2k-2}$$
where $Q\Sigma^{-1}X=\Omega QX$. Similarly, for an arbitrary positive integer $k$ the mapping $\Omega^kJ_\R:\Omega^k_0SO\to Q_0S^{-k}$ give rise to an infinite loop map $j^\R_k:Q\Omega^k_0SO\to Q_0S^{-k}$. Looping once, we obtain
$$\Omega j^\R_k:\Omega_0 Q\Omega^k_0SO\to Q_0S^{-k-1}.$$
Calculating $H_*(Q_0\Sigma^{-1}\Omega^k_0SO;\Z/2)$, for $k\equiv0,1,3,7\mathrm{\ mod\ }8$, and $H_*(Q\Sigma^{-1}BU;\Z/2)$ is quite straightforward (see Section 8). We may define classes in $H_*(Q_0S^{-k-1};\Z/2)$ by considering the image of the algebra generators of $H_*(Q_0\Sigma^{-1}\Omega^k_0SO;\Z/2)$ and $H_*(Q\Sigma^{-1}BU;\Z/2)$.

\begin{mthm}
If $k\leqslant7$ then $(\Omega j_k^\R)_*$ define infinitely many nontrivial classes in $H_*(Q_0S^{-k-1};\Z/2)$ in each dimension when $k\equiv0,1,3,7\mathrm{\ mod\ }8$. Similarly, $(\Omega j^\C_k)_*$ defines infinitely many nontrivial classes in $H_*(Q_0S^{-k-1};\Z/2)$ for $k\leqslant3$ odd.
\end{mthm}

Notice that `infinitely many nontrivial classes' is for the whole homology ring, in each dimension there will be only finitely many of such classes are nontrivial. The details and description of these classes are left to the proof. We suspect that similar work can be done for odd primes. At the last section we will comment on the equivariant versions of these results.

\textsc{acknowledgements.} I am in debt to Prof. Fred Cohen for many communications on the material and on the topic of this paper. I am grateful to Prof. Peter Eccles for many discussions on the topic, for reading the draft of this paper, and for his comments which led to improvements in this document. I have been a long term visitor at the University of Manchester during the time that I did this work, and I am grateful for the School of Mathematics at the University of Manchester for the hospitality. And thankyou to my family for their support.

\section{The fibre of the Adams operations}
The homological degree of a stable map $S^0\to S^0$ yields $\pi_0QS^0\simeq\pi_0^S\simeq\Z$. Let $Q_dS^0$ be the component of stable maps $S^0\to S^0$ of degree $d$. The loop sum with a stable map of degree $-d$ provides the translation map $*[-d]:Q_dS^0\to Q_0S^0$ resulting in homotopy equivalence between different components of $QS^0$. By definition the space $QS^0$ is an infinite loop space, as well as its base-point component $Q_0S^0$, where the addition is just the addition of loops. The space $Q_1S^0$ admits another loop space structure, once we consider the operation of composition of maps of degree $1$. We write $SG$ for $Q_1S^0$ when regarded as an infinite loop space with the composition product.\\
At each prime $p$, we have Sullivan's splitting of the space $SG$ \cite[Theorem 5.5]{8} as a product
$$SG\simeq J\times\cok J$$
whereas unlike the odd primes, at the prime $2$ this splitting is not an splitting of infinite loop spaces \cite{8}, \cite[Theorem 4]{16}. Nevertheless, by homotopy equivalence (of spaces and not of infinite loop spaces) $*[-1]:SG=Q_1S^0\to Q_0S^0$ we have the splitting
$$Q_0S^0\to J\times \cok J.$$
The spaces $J$ and $\cok J$ are relate to the real and complex $J$-homomorphisms throughout the affirmative solution of the Adams conjecture \cite[Theorem 1.1]{12}, \cite{1}. For $p=2$ the space $J$ is defined to be the the $2$-local fibre of the Adams operation $\psi^3-1:BSO\to BSO$, whereas the space $\cok J$ is defined to be the fibre of a certain map $G/O\to BSO$.\\
The real $J$-homomorphism is an infinite loop map $SO\to SG$ \cite{14}. The Adams conjecture predicts that the row in the following diagram is null-homotopic
$$\xymatrix{          &  G/O\ar[d]\\
BSO\ar[r]^-{\psi^3-1}\ar@{.>}[ru]^-{\beta} &  BSO\ar[r]^-{BJ} & BSG}$$
where $\beta$ is shown to exist by Quillen \cite[Theorem 1.1]{12} and is referred to as `a solution to the Adams conjecture'. Any solution to the Adams conjecture provides a factorisation of the $J$-homomorphism, up to homotopy, as \cite{13}
$$\xymatrix{SO\ar[r]^-\partial&  J\ar[r]^-\alpha        & SG}$$
where $\partial$ is the boundary map in the fibration sequence used to define $J$. At odd primes, the space $J$ is defined to be the $p$-local fibre of the Adams operation $\psi^q-1:BU\to BU$ where $q$ is a prime number chosen as explained at the previous section. There, again we have a similar commutative diagram and a factorisation of the $J$-homomorphism. Notice that using the translation map, at each prime $p$ we have homotopy factorisations
$$\begin{array}{lllll}
{J_\R}: &SO\stackrel{\partial}{\longrightarrow} J\stackrel{\alpha}{\longrightarrow} Q_0S^0, & &
{J_\C}: &U\stackrel{\partial}{\longrightarrow} J\stackrel{\alpha}{\longrightarrow} Q_0S^0
\end{array}$$
Since $\im(J_\C)_\#\simeq{_p\pi_*} J$ for $p>2$ \cite[Theorem 1.5.19]{26} then at odd primes it is enough to use the complex $J$-homomorphism, replacing $SO$ by $U$, and $J_\C$ by $J_\R$. At odd primes, the space $\cok J$ can be defined in a similar way as fibre of a certain map. For our purpose, the important is that at any prime $p$ the mapping $\alpha$ is a split monomorphism in the $p$-local category, which is one of the mappings used to obtain Sullivan's splitting of $Q_0S^0$, and hence it induces an (split) injection in $\Z/p$-homology.\\
Applying the iterated loop functor $\Omega^k$ to Sullivan's splitting, we obtain the splitting of $k$-fold loop spaces
\begin{equation}
\Omega^kQ_0S^0=QS^{-k}\to \Omega^k(J\times\cok J)\simeq \Omega^kJ\times\Omega^k\cok J
\end{equation}
as well as the following homotopy factorisations
$$\begin{array}{lllll}
{\Omega^kJ_\R}: &\Omega^k_0SO\stackrel{\Omega^k\partial}{\longrightarrow} \Omega^kJ\stackrel{\Omega^k\alpha}{\longrightarrow} \Omega^k_0Q_0S^0, & &
{\Omega^kJ_\C}: &\Omega^k_0U\stackrel{\Omega^k\partial}{\longrightarrow} \Omega^k_0J\stackrel{\Omega^k\alpha}{\longrightarrow} \Omega^k_0Q_0S^0
\end{array}$$
for  $p=2$, and $p>2$ respectively. In this case, all of our mappings are loop maps, i.e. in homology they induce multiplicative maps. Moreover, for a similar reason as above the mapping $\Omega^k\alpha$ induces a split injection in $\Z/p$-homology.

\section{Proof of Theorem $1$}
According to \cite[Proposition 1.5.22]{26} at the prime $2$ we have an injection $\im(J_\R)_\# \rightarrowtail{_2\pi_*}J$ which is an isomorphism, apart from dimensions $8i+1$ and $8i+2$ where in these dimensions there are element $\mu_{8i+1},\eta\mu_{8i+1}$ which are not given by the $J$-homomorphism. We break the proof of Theorem $1$ into two lemmata.

\begin{lmm}
Let $\mu_{8i+1}\in{_2\pi_{8i+1}}J$ be the element of order $2$ which is not in $\im(J_\R)_\#$. Then both $\mu_{8i+1}$ and $\eta\mu_{8i+1}$ map trivially under the composite ${_2\pi_*} J\rightarrowtail{_2\pi_*}Q_0S^0\stackrel{h}{\longrightarrow} H_*(Q_0S^0;\Z/2)$.
\end{lmm}

\begin{proof}
Let $\sigma_{8i-1}$ be the generator of ${_2\pi_{8i-1}}J\simeq\Z/2^j$ where $8i=2^j(2s+1)$. Then $\alpha_{8i-1}:=2^{j-1}\sigma_{8i-1}$ has order $2$ in this group.\\
According to Adams \cite[Theorem 12.13]{0} we may construct $\mu_{8i+1}$ as a triple Toda bracket $\{\eta,2,\alpha_{8i-1}\}$ where $\eta\in\pi_1^S$ is the Hopf invariant one element. We think of $\alpha_{8i-1}$ as a mapping $S^{8i-1}\to Q_0S^0$. We have the composite
$$S^{8i}\stackrel{\eta}{\longrightarrow} S^{8i-1}\stackrel{2}{\longrightarrow}S^{8i-1}\stackrel{\alpha_{8i-1}}{\longrightarrow}Q_0S^0$$
where the successive compositions are trivial. Hence, we obtain
$$S^{8i+1}\stackrel{\eta^\flat}{\longrightarrow} C_2\stackrel{\alpha'_\natural}{\longrightarrow}Q_0S^0$$
representing the triple Toda bracket for $\mu_{8i+1}$ where $\alpha'=\alpha_{8i-1}$ and $C_2$ is the mapping one of $2$. Note that $C_2$ has its top cell in dimension $8i$, hence $\eta^\flat$ is trivial in homology. Note that the mapping $\eta^\flat$ is in the stable range, hence can be taken as a genuine mapping. Therefore, the above composite is trivial in homology. Since the triviality in homology occurs only for dimensional reasons, then the indeterminacy in choosing the mappings $\eta^\flat$ and $\alpha_\natural$ will not make any change to this fact. Hence, $\mu_{8i+1}$ is trivial in homology, i.e. it maps trivially under the Hurewicz homomorphism ${_2\pi_*}Q_0S^0\to H_*(Q_0S^0;\Z/2)$. Consequently, $\eta\mu_{8i+1}$ is trivial in homology (although this was trivial from the given decomposition as well). This completes the proof.
\end{proof}

Next we consider the $\im(J_\R)_\#$.

\begin{lmm}
The image of ${_2\pi_*}SO\stackrel{(J_\R)_\#}{\rightarrowtail}{_2\pi_*}Q_0S^0\stackrel{h}{\longrightarrow}H_*(Q_0S^0;\Z/2)$ consists of only the image of the Hopf invariant one elements.
\end{lmm}

First, we fix some notation. Recall that for a given space $X$, if $x\in H_*X$ is spherical then it is primitive, and $x$ is $A$-annihilated, i.e. $Sq^r_*x=0$ for all $r>0$ where $Sq^r_*$ is dual to $Sq^r$.\\
It is quite standard that $H_*(SO;\Z/2)\simeq E_{\Z/2}(s_i:i\geqslant 1)$. The primitive classes in this ring are unique classes $p_{2n+1}^{SO}=s_{2n+1}+\sum s_is_{2n+1-i}$. Define primitive classes $p_{2n+1}^{S^0}\in H_{2n+1}(Q_0S^0;\Z/2)$ by $p_{2n+1}^{S^0}=(J_\R)_*p_{2n+1}^{SO}$. The action of $Sq^r_*$ on these classes \cite{14}, see also \cite[Lemma 3.6]{40}, is given by
$$\begin{array}{lll}
Sq^r_*p_{2n+1}^{SO}={2n+1-r\choose r}p_{2n+1-r}^{SO}, & Sq^r_*p_{2n+1}^{S^0}={2n+1-r\choose r}p_{2n+1-r}^{S^0}.
\end{array}$$
Therefore, the only $A$-annihilated primitive classes in $H_*(SO;\Z/2)$ are $p_{2^t-1}^{SO}$. Note that $(J_\R)_*$ is a monomorphism because of the Sullivan decomposition, and hence $p^{SO}_{2^t-1}$ maps isomorphically onto $p_{2^t-1}^{S^0}$.

\begin{proof}[of Lemma 3.2]
It is enough to show that the image of ${_2\pi_*}SO\stackrel{h}{\longrightarrow}H_*(SO;\Z/2)\stackrel{(J_\R)_*}{\rightarrowtail}H_*(Q_0S^0;\Z/2)$ only consists of the Hopf invariant one elements.\\
If $p_{2^t-1}^{SO}$ is spherical so $p_{2^t-1}^{S^0}$ is. On the other hand $\sigma^{2^t-1}_*p_{2^t-1}^{S^0}=g_{2^t-1}^2\in H_*(QS^{2^t-1};\Z/2)$ with $\sigma^{2^t-1}_*$ being the iterated homology suspension $H_*(Q_0S^0;\Z/2)\to H_{*+2^t-1}(QS^{2^t-1};\Z/2)$. Therefore $g_{2^t-1}^2\in H_*(QS^{2^t-1};\Z/2)$ should be spherical. However, according to Boardman and Steer \cite[Corollary 5.15]{41} (see also \cite[Proposition 5.8]{42}), the class $g_{2^t-1}^2$ is spherical if and only if there exists a Hopf invariant one element in ${_2\pi_{2^t-1}^S}$. This then implies that $t=1,2,3$. This the completes the proof.
\end{proof}

We then have completed the proof of Theorem $1$.

\section{The Hopf invariant one result}
The following should be well known, but we provide a proof.
\begin{lmm}
The Hopf invariant one elements in ${_2\pi_*^S}$ belong to the image of the real $J$-homomorphism. Consequently, the following statements are equivalent: $(1)$ There exists a Hopf invariant one element; $(2)$ $p_{2^t-1}^{SO}$ is spherical; $(3)$ $p_{2^t-1}^{S^0}$ is spherical.
\end{lmm}

\begin{proof}
The $J$-homomorphism is the stablised version of homomorphisms $\pi_rSO(q)\to\pi_{r+q}S^q$. The mapping $f:S^{2n-1}\to S^n$ has Hopf invariant one, if and only if $S^{n-1}$ is parallelisable, i.e. there exists a trivilisation of its tangent bundle. If $S^{n-1}$ is parallelisable we define  $\phi:S^{n-1}\to SO(n)$ by $x\mapsto [x,v_2,\ldots,v_n]$ where $\{v_2,\ldots,v_n\}$ is an $(n-1)$-frame orthogonal to $x$. The mapping  $f:S^{n-1}\times S^{n-1}\to S^{n-1}$ defined by $f(x,y)=\phi(x)\cdot y$ has bidegree $(1,1)$. The Hopf construction then yields a mapping $S^{2n-1}\to S^n$ of Hopf invariant one. Hence, the Hopf invariant one elements belong to the image of $J_\R$.\\
According to the proof of Lemma 3.2, $(3)$ implies $(1)$ and hence $(2)$. Recall from proof of Lemma 3.2 that $S^{2n-1}\to S^n$ has Hopf invariant one if and only if the adjoint map $S^{2n-2}\to\Omega S^n$ is nontrival in homology, hence $S^{n-1}\to\Omega^nS^n$ is nontrivial in homology. But it then the pull back $f':S^{n-1}\to SO(n)$ has to be nontrivial in homology, i.e. $hf'=p_{2^t-1}^{SO}$ for some $t$. This later maps to $p_{2^t-1}^{S^0}$. This completes the proof.
\end{proof}

The following is immediate from the Bott periodicity, and description of the homology suspension $\sigma_*:H_*\Omega^{k+1}_0SO\to H_{*+1}\Omega^k_0SO$ \cite{3}.
\begin{lmm}
Suppose $x\in H_*\Omega^k_0SO$ for some $k\geqslant 0$. Then $\sigma_*^tx=0$ for some $t\leqslant 7$.
\end{lmm}

We then have the following short proof of the Hopf invariant one result.

\begin{crl}
There are no Hopf invariant one elements in ${_2\pi_{2^t-1}^S}$ for $t>3$.
\end{crl}

\begin{proof}
Suppose there exist a Hopf invariant one element in ${_2\pi_{2^t-1}^S}$ with $t>3$. Then, according to Lemma 4.1, $p_{2^t-1}^{SO}\in H_{2^t-1}(SO;\Z/2)$ is spherical. The spherical class $p_{2^t-1}^{SO}$ desuspends to spherical classes $\xi_k\in H_{2^t-1-k}(\Omega^k_0SO;\Z/2)$ for any $k\leqslant 2^t-1$ with the property that $\sigma_*^k\xi_k=p_{2^t-1}^{SO}$. As $t>3$ we then may choose $k>7$, but this contradicts the previous lemma. This completes the proof.
\end{proof}

\section{The homology of iterated loops of $J$-homomorphisms}
The fact that $\Omega^k\alpha$ is an split injection in $\Z/p$-homology defines a set of generators in $H_*(Q_0S^{-k};\Z/p)$ in a geometric way, if we have a geometric description of $\alpha$ which is easy to understand. The next option is to show that $\Omega^kJ_\R$ and $\Omega^kJ_\C$ induce monomorphisms in homology, at least on certain elements. This is done by analysing well-known facts about $H_*(\Omega^kJ;\Z/p)$. We deal with $p=2$ and $p>2$ separately.
\subsection{The case $p=2$: Proof of Theorem $2$}
Let us write $H_*$ for $H_*(-;\Z/2)$. At the prime $2$ we have the fibration sequence
$$\xymatrix{\Omega^{k+1}_0BSO=\Omega^k_0SO\ar[r]^-{\Omega^k\partial}& \Omega^k_0J\ar[r]&\Omega^k_0BSO\ar[r]^-{\psi^3-1}&\Omega^k_0BSO.}$$
The homology of $\Omega^kJ$ is obtained by applying the Serre spectral sequence to the fibration $\Omega^k_0SO\stackrel{\Omega^k\partial}{\longrightarrow}\Omega^k_0J\longrightarrow\Omega^k_0BSO$ or passing to connected covers and then applying the Serre spectral sequence if necessary. Recall from \cite{3}, \cite{4} that
$$\begin{array}{lll}
H^*SO    &\simeq& \Z/2[f_{2i-1}:i\geqslant 1],\\
H^*BO    &\simeq& \Z/2[w_i:i\geqslant 1],\\
H_*BO    &\simeq& \Z/2[a_i:i\geqslant 1],\\
H_*SU/SO &\simeq& \Z/2[u_2,u_3,u_5,\ldots,u_{2n+1},\ldots].
\end{array}$$
Moreover, given a positive integer $j$ we define $\nu_j$ to be the largest positive integer that $2^{\nu_j}$ divides $3^{4j}-1$. The result of Cohen and Peterson \cite[Theorem 1.6]{2} on $H_*\Omega^kJ$ reads as following.
\begin{thm}
$(1)$ If $k\equiv 2,4,5,6$ mod $8$ then as vector spaces we have
$$H^*\Omega^k_0J\simeq H^*\Omega^kBSO\otimes H^*\Omega^{k+1}BSO.$$
$(2)$ If $k\equiv 1$ mod $8$ then as an algebra
$$H^*\Omega^k_0J\simeq H^*Spin\otimes H^*SO/U.$$
$(3)$ If $k\equiv 3$ mod $8$ then as an algebra
$$H^*\Omega^k_0J\simeq H^*SU/Sp\otimes H^*BSp.$$
$(4)$ If $k=8j-1$ mod $8$ then as an algebra
$$H_*\Omega^k_0J\simeq\Z/2[u_{2n+1}^2:n\equiv 0\mathrm{\ mod\ }2^{\nu_j}]\otimes\Z/2[u_2,u_{2n+1}:2n\not\equiv 0\mathrm{\ mod\ }2^{\nu_j}]\otimes
(\otimes_{n\geqslant 1}\Z/2[a_n]/a_n^{2^{\nu_j}}).$$
$(5)$ If $k=8j$ then as an algebra
$$H^*\Omega^k_0J\simeq(\otimes_{n\geqslant 1}\Z/2[w_n]/w_n^{2^{\nu_j}})\otimes\Z/2[f^2_{k2^{\nu_j}-1}]\otimes\Z/2[f_{2n-1}:2n\not\equiv 0\mathrm{\ mod\ }2^{\nu_j}].$$
\end{thm}
\begin{proof}[of Theorem 2]
First let $k\not\equiv 0,7$ mod $8$. Recall from \cite[Theorem 5.9]{17} that if $F\stackrel{i}{\to} E\to B$ is a fibration with $B$ simply connected and path connected, the mapping $i^*:H^qE\to H^qF$ may be calculated as the composite
$$H^qE\to E^{0,q}_\infty=E_{q+1}^{0,q}\subseteq E^{0,q}_q\subseteq\cdots\subseteq E^{0,q}_2=H^qF.$$
There is a similar statement for homology spectral sequence. For $k\not\equiv 0,7$ mod $8$, Cohen and Peterson \cite[Theorem 1.6]{2} have shown that the Serre spectral sequence collapses, and there are no nontrivial extension problems. Hence, the mapping $(\Omega^k\partial)_*$ is an injection, or $(\Omega^k\partial)^*$ is a surjection which is the same as saying $(\Omega^k\partial)_*$ is an injection, as we work with field coefficients. Combining this with the fact that $(\Omega^k\alpha)_*$ is an injection in homology, then implies that $\Omega^kJ_\R:\Omega^k_0SO\to Q_0S^{-k}$ is an injection in $\Z/2$-homology. Hence, in this case, the classes $x^{-k}_i$ defined in Theorem $2$ are nontrivial.\\
Next, let $k=8j-1$ for some $j$. We have the fibration $BO\simeq\Omega^{8j}BSO\to\Omega^{8j-1}_0J\to\Omega^{8j-1}BSO$. Similar to the previous case, together with the calculations in Theorem 5.1 confirms that $(\Omega^k\partial)_*$ maps the elements $a_i\in H_iBO$ monomorphically into $H_*\Omega^k_0J$. Again the injectivity of $(\Omega^k\alpha)_*$ implies that $(\Omega^k J_\R)_*$ is a monomorphism on the elements $a_i\in H_iBO$. Moreover, the fact that the mapping $(\Omega^k\partial)_*$ is a multiplicative map together with the above description of $H_*\Omega^k_0J$ tells is about $\ker(\Omega^k\partial)_*$ and hence about $\ker(\Omega^kJ_\R)_*$. In particular, $(\Omega^{8j-1}\partial)_*a_n^{2^{\nu_j}}=0$ for each $n\geqslant 1$. \\
Finally, let $k=8j$ where we have the fibration
$$SO\simeq\Omega^{8j+1}_0BO\to\Omega^{8j}_0J\to BO.$$
We consider the following commutative diagram
$$\xymatrix{H_{*+1}BO\ar[r] & H_{*+1}\Omega^{8j-1}_0J\\
            H_*SO\ar[u]^-{\sigma_*}\ar[r] & H_*\Omega^{8j}_0J\ar[u]_{\sigma_*}.}$$
The elements $s_{i-1}\in H_{i-1}SO$ are determined uniquely by the behavior of the homology suspension, by
$$\sigma_*s_{i-1}=p_i^{BO}=\left\{\begin{array}{ll}
                  a_{2n+1}+\textrm{decomposable terms} &\textrm{if }i=2n+1,\\
                  (p_{2l+1}^{BO})^{2^{\rho(i)}}               &\textrm{if }i=2^{\rho(i)}(2l+1)
                  \end{array}\right.$$
where $p_{2i+1}^{BO}=a_{2i+1}$ modulo decomposable terms. Notice that in $(\Omega^{8j-1}\partial)_*$ maps the terms $a_n$ monomorphically, and kills terms of the form $a_n^{2^{\nu_j}}$. This implies that $\ker(\Omega^{8j}\partial)_*$ as an ideal of $H_*SO$ is determined by
$$\langle s_i:i=2^{\nu_j}k-1,k\geqslant0\rangle.$$
Hence, for $i\neq 2^{\nu_j}k-1$ we have
$$x^{-8j}_i\neq 0.$$
Recall from \cite{3} that for $k\equiv 0,1,2,4$ mod $8$, $H_*\Omega^kSO$ is exterior. This together with the fact that $(\Omega^kJ_\R)_*$ is multiplicative implies that $(x^{-k}_i)^2=0$ in these cases. Moreover, the Cartan formula $(Q^i\xi)^2=Q^{2i}\xi^2$
implies that $(Q^Ix^{-k}_i)^2=0$ for any nonempty sequence $I$. Hence, we see that the subalgebra of $H_*Q_0S^{-k}$ spanned by classes of the form $Q^Ix^{-k}_i$ is an exterior algebra if $k\equiv 0,1,2,4$ mod $8$. The action of the Steenrod operations $Sq^r_*$ on the classes $x^{-k}_i$ is determined by their actions on the generators of $H_*\Omega_0^kSO$. This together with the Nishida relations determines the action of these operations on the classes $Q^Ix^{-k}_i$. Finally, $x_i^{-k}$ is the image of a generator in $H_i\Omega^k_0SO$. Applying Lemma 4.2, then implies that $\sigma_*^tx^{-k}_i=0$ for some $t\leqslant7$.
\end{proof}

\subsection{Proof of Theorem 5: The case of odd primes }
Let us write $H_*$ for $H_*(-;\Z/p)$. At the prime $p>2$ we have the fibration sequence
$$\xymatrix{\Omega^{k+1}_0BU=\Omega^k_0U\ar[r]^-{\Omega^k\partial}& \Omega^k_0J\ar[r]&\Omega^k_0BU\ar[r]^-{\psi^q-1}&\Omega^k_0BU.}$$
In this case, the homology of $\Omega^k_0J$ is obtained by applying the Serre spectral sequence to the fibration
$$\Omega^k_0U\stackrel{\Omega^k\partial}{\longrightarrow}\Omega^k_0J\longrightarrow\Omega^k_0BU.$$
Notice that for the spaces in the complex Bott periodicity, i.e. $U$ and $BU$, we have \cite[Assertions 1,3]{3}
$$\begin{array}{lll}
H_*(U;\Z) &\simeq& E_{\Z}(u_{2i-1}:i\geqslant 1, u_{2i-1}\textrm{ primitive}),\\
H_*(BU;\Z)&\simeq& \Z[c_{2i}:i\geqslant 1].
\end{array}$$
Hence, at each prime $p$ the homology will have the `same' structure over $\Z/p$ with the same generators at the same dimensions. The homology of $\Omega^k_0J$ at odd primes is calculated, due to Cohen and Peterson \cite[Theorem 1.3, Theorem 1.4]{2}. We recall the result.
\begin{thm}
$(i)$ There is an isomorphism of algebras
$$H_*\Omega^{2n}_0J\to\bigotimes_{p|q^{n+i}-1}E_{\Z/p}((\Omega^{2n}\partial)_*u_{2i-1})\otimes\Z/p[c_{2i}]$$
where $c_{2i}$ is the unique element mapping to $c_{2i}\in H_{2i}\Omega^{2n}BU.$\\
$(ii)$ There is an isomorphism of algebras
$$H_*\Omega^{2n-1}_0J\to\bigotimes_{p|q^{n+i}-1}\Z/p[(\Omega^{2n-1}\partial)_*c_{2i}]\otimes E_{\Z/p}(u_{2i+1})$$
where $u_{2i+1}$ is the unique element mapping to $u_{2i+1}\in H_{2i+1}\Omega^{2n-1}_0U$.\\
In both cases the uniqueness of the generators comes from their definitions and from the Serre spectral sequence arguments respectively.
\end{thm}

The proof of Theorem $5$, at the case of odd primes, is similar to the proof of Theorem $2$. Applying calculation of $H_*\Omega^k_0J$ yields to information $\Omega^k J_\C$ and that in particular $(\Omega^kJ_\C)_*$ is injective on the generators. Note that the injectivity of $(\Omega^kJ_\C)_*$ is a consequence of the injectivity of $(\Omega^k\partial)$ in $\Z/p$-homology, which is clear in the above presentation. Hence, the iterated complex $J$-homomorphism $\Omega^kJ_\C$ defines nontrivial classes
$w_{2i+1-\epsilon}^{-k}\in H_*Q_0S^{-k}$ by
$$w_{2i+1-\epsilon(k)}^{-k}=\left\{\begin{array}{ll}
                             (\Omega^k J_\C)_*c_{2i}& \textrm{if $k$ is odd,}\\
                             (\Omega^k J_\C)_*u_{2i+1}& \textrm{if $k$ is even,}
                             \end{array}\right.$$
Moreover, when $k$ is even the homology ring $H_*\Omega^k_0U\simeq H_*U$ is exterior. This then means that
$$(w_{2i+1}^{-k})^2=0.$$
The ring $H_*Q_0S^{-k}$ is a bicommutative and biassociative Hopf algebra which are completely known \cite[Chpater III (8.7)]{99}. The classes $w_{2i+1}^{-k}$, and consequently the classes $Q^Iw_{2i+1}^{-k}$ are primitive. At $p>2$, this implies $(Q^Iw_{2i+1}^{-k})^2=0$ if it is an odd dimensional class. Hence, the classes $Q^Iw_{2i+1}^{-k}$ generate an exterior subalgebra in $H_*(Q_0S^{-k};\Z/p)$.\\
On the other hand, if $Q^Iw_{2i+1}^{-k}$ is even dimensional then the classes $Q^Iw_{2i+1}^{-k}$ either have to generate a polynomial algebra or a truncated algebra of truncation height $p^n$. However, we have the Cartan formula $(Q^i\xi)^p=Q^{pi}\xi^p$. This implies that $(Q^Iw_{2i+1}^{-k})^p=0$. Hence, the classes $Q^Iw_{2i+1}^{-k}$, with $I$ odd dimensional, generate a truncated polynomial subalgebra of truncated height $p$ inside $H_*(Q_0S^{-k};\Z/p)$.\\
The action of the mod $p$ Steenrod operations on the classes $w_{2i+1-\epsilon}^{-k}$ coupled with the Nishida relations tells us the action of the operation $Sq^r_*$ on the classes $Q^Iw_{2i+1-\epsilon}^{-k}$ and hence on the subalgebra of $H_*Q_0S^{-k}$ spanned by these classes. Finally, note that the behavior of $w_{2i+1}^{-k}$ under the homology suspension is determined by the behavior of the homology suspension $\sigma_*:H_*\Omega^{k+1}_0U\to H_{*+1}\Omega^k_0U$. The following is immediate from \cite{3} (see also \cite{6}).
\begin{lmm}
Let $p\geqslant 2$ be a prime, and $x\in H_*(\Omega^k_0U;\Z/p)$ for some $k\geqslant0$. Then $\sigma_*^tx=0$ for some $t\leqslant3$.
\end{lmm}
This lemma then shows that $\sigma_*^tw_{2i+1}^{-k}=0$ for some $t\leqslant3$, if $k>3$. This completes the proof of Theorem $5$ at the case of odd primes.

\subsection{The $\Z/2$-homology of $\Omega^kJ_\C$: Proof of Theorem $5$ when $p=2$.}
We write $H_*$ for $H_*(-;\Z/2)$. The mapping $J_\C$ factors through $J_\R$ as $J_\C:U\to SO\to Q_0S^0$. For $0\leqslant t\leqslant 7$, applying the functor $\Omega^{8k+t}$ to this factorisation and restricting to the base point components, we obtain
$$\begin{array}{llllllllll}
U & \stackrel{\iota_0}{\longrightarrow} & \Omega^{8k}_0SO     & \simeq SO       & \longrightarrow & \Omega^{8k}_0J\\
BU& \stackrel{\iota_1}{\longrightarrow} & \Omega^{8k+1}_0SO   & \simeq SO/U     & \longrightarrow & \Omega^{8k+1}_0J\\
U & \stackrel{\iota_2}{\longrightarrow} & \Omega^{8k+2}_0SO   & \simeq U/Sp     & \longrightarrow & \Omega^{8k+2}_0J\\
BU& \stackrel{\iota_3}{\longrightarrow} & \Omega^{8k+3}_0SO   & \simeq BSp      & \longrightarrow & \Omega^{8k+3}_0J\\
U & \stackrel{\iota_4}{\longrightarrow} & \Omega^{8k+4}_0SO   & \simeq Sp       & \longrightarrow & \Omega^{8k+4}_0J\\
BU& \stackrel{\iota_5}{\longrightarrow} & \Omega^{8k+5}_0SO   & \simeq Sp/U     & \longrightarrow & \Omega^{8k+5}_0J\\
U & \stackrel{\iota_6}{\longrightarrow} & \Omega^{8k+6}_0SO   & \simeq U/O      & \longrightarrow & \Omega^{8k+6}_0J\\
BU& \stackrel{\iota_7}{\longrightarrow} & \Omega^{8k+7}_0SO   & \simeq BO       & \longrightarrow & \Omega^{8k+7}_0J.\\
\end{array}$$
The calculation of $(\Omega^k J_\C)_*$ then reduces to understanding the homology of $\iota_t$ where $0\leqslant t\leqslant 7$. We have the following observation.
\begin{lmm}
For $t=5,6$ we have $(\iota_t)_*=0$. For $t\neq5,6$ the nontrivial action of $(\iota_t)_*$ is determined as following (modulo decomposable terms)
$$\begin{array}{llllllll}
(\iota_0)_*u_{2i+1}&=&s_{2i+1}, & & (\iota_3)_*c_{4i}  &=&p_{4i},\\
(\iota_1)_*c_{2i}  &=&c_{2i}, & &   (\iota_4)_*u_{4i-1}&=&z_{4i-1},\\
(\iota_2)_*u_{4i+1}&=&u_{4i+1}, & & (\iota_7)_*c_{2i}  &=&a_i^2,
\end{array}$$
where $c_{2i}\in H_{2i}SO/U$, $z_{4i-1}\in H_iSP$, $p_{4i}\in H_{4i}BSp$, and $a_i\in H_iBO$ are generators.
\end{lmm}
Notice that the generators $u_{2i+1}\in H_{2i+1}U$ are primitive, hence their image under $(\Omega^kJ_\C)_*$ for $k$ even is also a primitive element which tells one how to calculate the decomposable terms.
\begin{proof}
We only show that $(\iota_t)_*=0$ for $t=5,6$ and the other cases are straightforward. Notice that $\iota_6$ is a loop map, and in homology its domain is an exterior algebra whereas its target is a polynomial algebra, hence $(\iota_6)_*=0$. Next, consider $\iota_5:BU\to Sp/U$. Recall that $H^*Sp/U$ is an exterior algebra with generators in dimensions $4i+2$ \cite[Assertaion 27]{3} whereas $H^*BU$ is a polynomial algebra. This implies that $(\iota_5)^*=0$, and hence $(\iota_5)_*=0$. This completes the proof.
\end{proof}

The above lemma shows that $(\Omega^kJ_\C)_*=0$ if $k\equiv 5,6$ mod $8$. Recall that $w_{2i+1-\epsilon}^{-k}\in H_*Q_0S^{-k}$ in Theorem 2 are defined by
$$w_{2i+1-\epsilon}^{-k}=\left\{\begin{array}{ll}
                             (\Omega^k J_\C)_*c_{2i}& \textrm{if $k$ is odd,}\\
                             (\Omega^k J_\C)_*u_{2i+1}& \textrm{if $k$ is even,}
                             \end{array}\right.$$
where for $k$ odd we have $\Omega^k J_\C:BU\to Q_0S^{-k}$ and for $k$ even $\Omega^k J_\C:U\to Q_0S^{-k}$. We then have the following observation.
\begin{crl}
The image of $(\Omega^kJ_\C)_*$ in $\Z/2$-homology is determined as following: The classes $w_{2i+1-\epsilon}^{-k}$ are nontrivial only in one of the following cases: (1) If $k\equiv 0,1,7$ mod $8$; (2) If $k\equiv 2,3$ mod $8$ and $i$ is even; (3) If $k\equiv 4$ mod $8$ and $i$ is odd.
\end{crl}
Lemma 5.3 shows that $\sigma_*^tw_{2i+1}^{-k}=0$ for some $t\leqslant3$, if $k>3$. Since $H_*U$ is an exterior algebra then for $k\equiv 0,2,4$ the subalgebra of $H_*Q_0S^{-k}$ spanned by the classes of the for $Q^Iw^{-k}_{2i+1-\epsilon}$ is an exterior algebra. The action of the Steenrod operation is also similar. These observations complete the proof of Theorem 5 when $p=2$.

\subsection{An application to the complex transfer}
The following will be used here, and later on.
\begin{lmm}
Suppose $\alpha\in{_2\pi_k^S}$ is a generator. Then the stable adjoint of $\alpha$ as an element of ${_2\pi_1}Q_0S^{-k+1}$ maps nontrivially under the Hurewicz homomorphism
$$h^1:{_2\pi_1}Q_0S^{-k+1}\to H_1Q_0S^{-k+1}.$$
\end{lmm}

\begin{proof}
Consider the following commutative diagram
$$\xymatrix{
{_2\pi_1}Q_0S^{-k+1}\ar[r] & H_1Q_0S^{-k+1}\\
{_2\pi_0}QS^{-k}\ar[r]\ar[u]^-{\simeq} & H_0QS^{-k}.\ar[u]_-{\sigma_*}}$$
The fact that $\alpha\in{_2\pi_k^S}$ is a generator implies that it maps nontrivially under the Hurewicz map $h^0:{_2\pi_0}QS^{-k}\to H_0QS^{-k}$, representing a class with odd coefficient in $H_0(Q_0S^{-k};\Z)$. Hence, if $\sigma_*h^0\alpha\neq0$ then it will have odd coefficient in $H_*(Q_0S^{-k+1};\Z)$. On the other hand, notice that $Q_0S^{-k+1}$ is a loop space which means its fundamental group is Abelian, so $\pi_1Q_0S^{-k+1}\simeq H_1(Q_0S^{-k+1};\Z)$. Hence, $\sigma_*h^0\alpha$ has odd coefficient in $H_1(Q_0S^{-k+1};\Z)$, therefore its $\Z/2$-reduction is nontrivial, i.e. $\alpha$ maps nontrivially under the Hurewicz homomorphism
$${_2\pi_1}Q_0S^{-k+1}\to H_1(Q_0S^{-k+1};\Z/2).$$
\end{proof}

Next, we apply the $\Z/2$-homology of $\Omega^kJ_\C$ to give a proof of the following well known fact in terms of homology.
\begin{lmm}
The complex transfer
$$\lambda_\C:Q\Sigma\C P_+\longrightarrow Q_0S^0$$
does not induce an epimorphism on the $2$-primary homotopy groups.
\end{lmm}

\begin{proof}
The complex transfer is the infinite loop extension of the composite
$$\xymatrix{\Sigma\C P_+\ar[r] & U\ar[r]^-{J_\C} & Q_0S^0.}$$
Consider the generator $\sigma_{8i-1}\in{_2\pi_{8i-1}}J$. If $\sigma_{8i-1}$ is in the image of the complex transfer then it admits a factorisation
$\xymatrix{S^{8i-1}\ar[r]^-{\sigma'_{8i-1}} & Q\Sigma\C P_+\ar[r]^-{\lambda_\C} & Q_0S^0}$
which adjoints down to
$$\xymatrix{S^1\ar[r] & \Omega^{8i-2}Q\Sigma\C P_+\ar[r]^-{\Omega^{8i-2}\lambda_\C} & Q_0S^{-8i+2}}.$$
According to the previous lemma this composition should be nontrivial in $\Z/2$-homology, whereas Lemma 5.4 implies that $(\Omega^{8i-2}\lambda_\C)_*=0$. This is a contradiction. Hence,
$$(\lambda_\C)_\#:{_2\pi_*^S}\Sigma\C P_+\simeq{_2\pi_*}Q\Sigma\C P_+\longrightarrow{_2\pi_*}Q_0S^0\simeq{_2\pi_*^S}$$
is not an epimorphism.
\end{proof}

\section{Spherical classes in $H_*(\Omega^k_0J;\Z/2)$: Proof of Theorem $3$}
Throughout this section we write $H_*$ for $\Z/2$-homology. We determine the spherical classes in $H_*\Omega^k_0J$. This in particular determines which of the classes $x^{-k}_d$ defined in Theorem 2 are spherical. First, we record the following observation on spherical classes in $H_*\Omega^k_0SO$, which should be well-known.
\begin{lmm}
Let $x\in H_d\Omega^{8t+k}_0SO$ be spherical with $hf=x$ for some $f\in{_2\pi_d}\Omega^{8t+k}_0SO$. Then one of the following cases holds:\\
$(1)$ $k=0$ and $d=1,3,7$; $(2)$ $k=1$ and $d=2,6$; $(3)$ $k=2$ and $d=1,5$; $(4)$ $k=3,4,5$ and $d=7-k$; $(6)$ $k=6$ and $d=1,2$; $(7)$ $k=7$ and $d=1,2,4,8$.\\
In any of these cases, the mapping $f$ maps to one of the Hopf invariant one elements under the composite
$${_2\pi_d}\Omega^{8t+k}_0SO\stackrel{\simeq}{\longrightarrow}{_2\pi_d}\Omega^k_0SO\stackrel{\simeq}{\longrightarrow}
{_2\pi_{d+k}}SO\stackrel{J_\R}{\longrightarrow}{_2\pi_{d+k}^S.}$$
\end{lmm}

\begin{proof}
By Bott periodicity, it is enough to prove the lemma when $t=0$. The lemma is an outcome of the Adams's Hopf invariant one result.\\
If $k=0$ this is just Lemma 4.1. For $k=7$, $a_1\in H_1BO$ is spherical by inclusion of the bottom cell whereas a spherical class in $H_d\Omega^7_0SO\simeq BO$ with $d>1$ desuspends to a spherical class in $H_{d-1}SO$. Therefore $d=2,4,8$.\\
Other cases, also follow similarly by using the homology suspension or desuspension arguments. For instance, if $k=1$, then a spherical class $x\in H_dSO/U$ lives in an exterior algebra which means that it cannot be a square, i.e $x$ determines a nontrivial indecomposable element in $H_*SO/U$ which suspends to a spherical class in $H_{d+1}SO$. Bearing in mind that $SO/U$ has its bottom cell in dimension $2$ proves our claim.\\
For $k=2$, notice that $H^*SO/U$ is polynomial. The Eilenberg-Moore spectral sequence argument then shows that the homology suspension maps indecomposable elements of $H_*U/Sp$ isomorphically onto primitive elements in  $H_{*+1}SO/U$. Moreover, notice that $H_*U/Sp$ is exterior. Hence, a spherical class in $H_dU/Sp$ must correspond to an indecomposable element which means that it maps isomorphically to a spherical class in $H_{d+1}SO/U$. This implies that $d=1,5$.\\
For $k=3$, a spherical class in $x\in H_dBSp$, with $d>4$, desuspends to a spherical class in $H_{d-4}BO$. This implies that $d-4\leqslant8$, i.e. $d\leqslant 12$. On the hand the $A$-annihilated primitive classes in $H_*BSp$ live in dimensions $4(2^i-1)$. This implies that only $4$ and $12$ dimensional classes in $H_*BSp$ have chance to be spherical. If there is a spherical class in $H_{12}BSp$ it then must desuspend to a primitive class in $H_8BO$. However, the only primitive class in $H_8BO$ is given by $a_1^8$ which dies under the homology suspension. Hence, we are only left with a spherical class in dimension $4$ for $H_*BSp$.\\
For $k=4$, again we observe that $H^*BSp$ is polynomial, hence a spherical class in $H_dSp$ must survives to a spherical class in $H_{d+1}BSp$, i.e. we must have $n=3$.\\
If $k=5$, a spherical class in $H_dSp/U$ with $d>2$ must desuspend to a spherical class in $H_{d-2}BO$ which mean $d-2\leqslant8$. A reasons simlar to previous cases leads to a contradiction. Hence, $d=2$ which is given by the inclusion of the bottom cell.\\
Finally, in the case of $k=6$ the inclusion of the bottom cell shows that $u_1$ is spherical, where as in dimension $2$ the spherical class is $u_1^2$ with $\sigma_*a_1=u_1^2$. We note that this is related to $\eta$, as delooping once more we will obtain a mapping $S^3\to Sp/U$ which is detected by $Sq^2$ on the bottom cell. The other candidates can be shown not to be spherical similar to the previous cases.
\end{proof}

We apply this observation to the following problem. From the splitting (1.2) it is clear that if a spherical class $x\in H_*\Omega^k_0J$ maps nontrivially to a spherical class in $H_*Q_0S^{-k}$, and projects back onto a spherical class in $H_*\Omega^k_0J$. Hence, in order to classify such spherical classes in $H_*Q_0S^{-k}$, it is enough to determine spherical classes in $H_*\Omega^k_0J$, i.e. it is enough to calculate the image of the Hurewicz homomorphism
$$h^d:{_2\pi_d}\Omega^k_0J\longrightarrow H_d\Omega^k_0J.$$
We first, deal with $\mu_{8i+1}$ and $\eta\mu_{8i+1}$ which are not in the image of the $J$-homomorphism. Notice that according to Lemma 5.6 they will map nontrivially under the Hurewicz homomorphism $h^1$.
\begin{lmm}
For the classes $\mu_{8i+1}$ and $\eta\mu_{8i+1}$ we have the following equations:
$$\begin{array}{llll}
(1) & h^4(\mu_{8i+1})&=&0,\\
(2) & h^2(\mu_{8i+1})&\neq& 0,\\
(3) & h^3(\eta\mu_{8i+1})&=&0,\\
(4) & h^2(\eta\mu_{8i+1})&\neq&0.
\end{array}$$
\end{lmm}

\begin{proof}
For the first equation, consider the composition
$$S^{8i}\stackrel{\eta}{\longrightarrow} S^{8i-1}\stackrel{2}{\longrightarrow}S^{8i-1}\stackrel{\alpha_{8i-1}}{\longrightarrow}Q_0S^0.$$
This adjoints down to
$$S^{3}\stackrel{\eta}{\longrightarrow} S^{2}\stackrel{2}{\longrightarrow}S^{2}\stackrel{\alpha''}{\longrightarrow}Q_0S^{-8i+3}$$
where $\alpha''$ is the adjoint of $\alpha_{8i-1}$. The successive compositions $2\eta$ and $\alpha''2$ are trivial. Hence, we obtain
$$S^{4}\stackrel{\eta^\flat}{\longrightarrow} C'_2\stackrel{\alpha''_\natural}{\longrightarrow}Q_0S^{-8i+3}$$
representing the adjoint of $\mu_{8i+1}$ as an element in ${_2\pi_4}Q_0S^{-8i+3}$. Notice that $(\eta^\flat)_*=0$ for dimensional reasons, hence $h^4(\mu_{8i+1})=0$.\\
On the other hand, according to Lemma 5.6 $\mu_{8i+1}:S^1\to Q_0S^{-8i}$ is detected by homology. Since $\mu_{8i+1}$ does not belong to the image of the $J$-homomorphism, hence it maps nontrivially under the mapping $\partial_\#:{_2\pi_1}\Omega^{8i}_0J\to{_2\pi_1}\Omega^{8i}_0BSO\simeq{_2\pi_1}BO$ where $\partial$ was introduced in Section 5.1. Applying Lemma 5.6 we observe that $h^1(\partial_\#\mu_{8i+1})=a_1$. Notice that in the composition
$$S^1\longrightarrow\Omega^{8i}_0J\longrightarrow\Omega^{8i}_0BSO\longrightarrow BO$$
which realises $\partial_\#\mu_{8i+1}$, all maps are loop maps. Hence, we may adjoint up once. In homology we then obtain
$$\sigma_*h^1(\partial_\#\mu_{8i+1})=\sigma_*a_1=u_1^2\in H_*U/O$$
hence showing that $h^2(\mu_{8i+1})=\sigma_*h^1(\mu_{8i+1})\neq 0$, where $H_*U/O\simeq\Z/2[u_{2i+1}:i\geqslant 0]$.\\

For the third equality, notice that $\mu_{8i+1}:S^{8i+1}\to S^0$ has stable adjoint $S^2\to Q_0S^{-8i+1}$. Hence, $\eta\mu_{8i+1}$ can be represented as
$$S^3\stackrel{\eta}{\longrightarrow} S^2\stackrel{\mu_{8i+1}}{\longrightarrow} Q_0S^{-8i+1}$$
where $\eta_*=0$. Hence, $h^3(\eta\mu_{8i+1})=0$.\\
Finally, note that adjointing down $S^3\to S^2\to Q_0S^{-8i+1}$ only once, we may represented $\eta\mu_{8i+1}$ as
$$S^2\stackrel{\eta}{\longrightarrow} \Omega S^2\stackrel{\Omega\mu_{8i+1}}{\longrightarrow} Q_0S^{-8i}$$
which shows that $h^2(\eta\mu_{8i+1})=(\Omega\mu_{8i+1})_*\eta_*g_2=(\Omega\mu_{8i+1})_*g_1^2=(h^1\mu_{8i+1})^2=a_1^2\in H_*BO$ whereas $\eta_*g_2=g_1^2$ as $\eta$ has Hopf invariant one. This equality then show that $h^2(\eta\mu_{8i+1})\neq 0$.
\end{proof}

We note that the above lemma does not tell us about $h^3(\mu_{8i+1})$, although we guess that $h^3(\mu_{8i+1})\neq 0$.

Now we can complete proof of Theorem $3$.
\begin{proof}[of Theorem $3$]
Assume that $x^{-k}_d$ is spherical. This project onto a spherical class in $H_*\Omega^k_0J$. Notice that the composite $\Omega^k_0SO\to\Omega^k_0J\to\Omega^kBSO$ is null homotopic, i.e. the class $x_d^{-k}$ maps trivially into $H_*\Omega^kBSO$. Hence, it can not arise from the classes $\mu_{8i+1}$ and $\eta\mu_{8i+1}$ (detailed calculated for this also can be carried out). On the other hand, $x_d^{-k}$ cannot arise from ${_2\pi_*}\cok J$. To see this note that Sullivans decomposition yields a (split) fibration
$$\Omega^k\cok J\longrightarrow \Omega^kQ_0S^0\longrightarrow\Omega^kJ.$$
If $x_d^{-k}$ does arise from ${_2\pi_*}\cok J$, it then must be trivial in $H_*\Omega^k_0J$ which is a contradiction. Hence, $x_d^{-k}$ is spherical only if it arises from the image of the $J$-homomorphism. Lemma 6.1 then finishes the proof.
\end{proof}

\section{Truncated subalgebras in $H_*(Q_0S^{-k};\Z/2)$:Proof of Proposition $4$}
We keep $H_*$ for $H_*(-;\Z/2)$ unless otherwise stated. Our use of the iterated loops of the real and complex $J$-homomorphism comes with a cost. In both of the real and complex case, we used to focus on the base point components of $\Omega^kSO$ and $\Omega^kU$. But there are some elements in $H_*Q_0S^{-k}$ which are not seen in this way, neither detecting them with the Eilenberg-Moore spectral sequence is the easiest way to do this.\\
As a motivating example, recall that \cite[Part I, Lemma 4.10]{6}
$$H_*Q_0S^0\simeq\Z/2[Q^Ix_i:\ex(Q^Ix_i)>0, (I,i)\textrm{ admissible}],$$
where $x_i=Q^i[1]*[-2]$ and, denotes the loop sum in $H_*QS^0$, and $I=(i_1,\ldots,i_r)$ is admissible if $i_j\leqslant2i_{j+1}$. More important for us us that $[n]$ is the image of $n\in\pi_0QS^0$ under the Hurewicz homomorphism $\pi_0QS^0\to H_0(QS^0;\Z).$ Note that in dimension $0$ both $\pi_0QS^0$ and $H_0QS^0$ are groups under the loop sum and the Hurewicz map respects this operation, i.e. $[n]*[m]=[n+m]$. The map $n:S^0\to QS^0$ extends to an infinite loop map $n:QS^0\to QS^0$, providing necessary maps for homotopy equivalence between different path components of $QS^0$.

\subsection{Some remarks on the geometry of $QS^{-k}$}
In general, we know that $\pi_0QS^{-k}\simeq\pi_k^S$ is a direct sum of cyclic groups of the form $\Z/2^d(2s+1)$, although we are far from knowing these summands in general. If we fix an automorphism of $\pi_k^S$, then we may label different components of $QS^{-k}$ by writing $Q_\alpha S^{-k}$ for $\alpha\in\pi_0QS^{-k}\simeq\pi_k^S$.\\
For $k=0$, this automorphism is given by the degree mapping. In fact, we have a mapping $\deg:QS^0\to\Z\simeq\Omega^\mathrm{fr}_0$ which induces the isomorphism $\pi_0QS^0\to\Z$. For $k>0$, a choice is provided by the framing homomorphism (up to homotopy)
$$\mathrm{fr}:QS^{-k}\to\Omega^\mathrm{fr}_k.$$
If we take $\gamma^k$ to be generator of a summand in $\pi_0QS^{-k}$ we then may represent it as an element in $QS^{-k}$. We then define a mapping
$[-\gamma^k]:Q_{\gamma^k}S^{-k}\to Q_0S^{-k}$ by $\alpha\mapsto\alpha-\gamma^k$. This induces a homotopy equivalence between different components of $QS^{-k}$.\\
If we consider the $2$-primary homotopy groups, we are then only concerned with $\Z/2^d$ part of each summand of the form $\Z/2^d(2s+1)$. If we write ${_2QS^{-k}}$ for the $2$-local version of $QS^{-k}$, we then may consider the projection map $QS^{-k}\to{_2QS^{-k}}$ as the identification map induced by $\Z/2^d(2s+1)\to\Z/2^d$. That is to say, in the level of spaces, for $\alpha,\beta$ belonging to the summand $\Z/2(2s+1)$ if $\alpha\equiv\beta$ mod ${_2\pi_k^S}$ then $Q_\alpha S^{-k}$ and $Q_\beta S^{-k}$ are identified under the projection map $QS^{-k}\to{_2QS^{-k}}$. Of course, ${_2QS^{-k}}$ is more complicated than this, but we allow ourselves to think of different components of ${_2QS^{-k}}$ as equivalence classes of $2$-local versions of components of $QS^{-k}$, and we drop the sub-preindex $2$ from the notation.\\
Next, the Hurewicz homomorphism $h:{_2\pi_0}QS^{-k}\to H_0QS^{-k}$. For any $\alpha\in{_2\pi_k^S}$, we define $[\alpha]=h\alpha\in H_0QS^{-k}$. Then it is clear that $[\alpha]*[\beta]=[\alpha+\beta]$ for $\alpha,\beta\in\pi_0QS^{-n}$ where $*$ denotes the Pontrjagin product in $H_*QS^{-k}$ induces by the loop sum. We note that $[0]$ acts as the identity element for the Pontrjagin product in $H_*QS^{-k}$. Applying the Cartan formula to the relation $[0]*[0]=[0+0]=[0]$ together with induction shows that $Q^i[0]=0$ for all $i>0$; whereas $Q^0[0]=[0]*[0]=[0]$.\\
For a given class $\alpha\in{_2\pi_k^S}$ we then may consider the homology classes $Q^I[\alpha]$. If $\alpha$ and $\beta$ are two generators for two summands of ${_2\pi_k^S}$ then the fact that there is no additive relation between these generators implies that there is no multiplicative relation between the classes of the form $Q^I[\alpha]$ and $Q^J[\beta]$. Finally, notice that the value of $Q^I[\alpha+\beta]=Q^I([\alpha]*[\beta])$ is calculated by the Cartan formula.\\

Proposition $4$, then determines the type of subalgebras in $H_*Q_0S^{-k}$ that one can obtain from elements of the form $\gamma^k\in{_2\pi_*^S}$ of order $2^d$. Before continuing with the proof we recall the Cartan formula, for $t>0$,
$$(Q^I\xi)^{2^t}=Q^{2^tI}\xi^{2^t}$$
where $2^tI=(2^ti_1,\ldots,2^ti_l)$ for $I=(i_1,\ldots,i_l)$.

\begin{proof}[of Proposition $4$]
Observe that $[\gamma^k]$ is the generator of $H_0Q_{\gamma^k}S^{-k}$ where $Q_{\gamma^k}S^{-k}$ denotes the component of $QS^{-k}$ corresponding to $\gamma^k$. By construction of the Kudo-Araki operations, $Q^i$,  applying an operation $Q^I$ with $l(I)=l$ to $[\gamma^k]$ we land in the component $2^l$ of $QS^{-k}$ say $Q_{2^l\gamma^k}S^{-k}$, i.e. $Q^I[\gamma^k]\in H_{\dim (I)}Q_{2^l\gamma^k}S^{-k}$. Hence, we use the translation isomorphism $*[-2^l\gamma^k]:Q_{2^l\gamma^k}S^{-k}\to Q_0S^{-k}$ to land in the base point component of $QS^{-k}$. Finally, notice that once we are in the base point component then applying operations will not change the component, as $20=0$ !\\
Since, we are concerned with a cyclic group it suffices to consider sequences of length at most $d$, that is we consider classes of the form
$$
\begin{array}{ll}
[\gamma^k_J]=Q^J[\gamma^k]*[-2^{l(J)}\gamma^k]\in H_{\dim(J)}Q_0S^{-k} & \textrm{if }l(J)<d,\\

[\gamma^k_J]=Q^J[\gamma^k]\in H_{\dim(J)}Q_0S^{-k} & \textrm{if }l(J)=d,
\end{array}
$$

Notice that for different choices of admissible sequences $J$, the classes $[\gamma^k_J]$ are not decomposable elements as they are only characterised by the sequence $J$. We then like to show that the classes $Q^I[\gamma^k_J]$ form a truncated polynomial algebra of truncation height not more than $2^d$. This comes from the Cartan formula mentioned above, and the fact $Q^I[0]=0$ for any nonempty sequence with a nonzero entry. More precisely, if $l(I)<d$ we have
$$\begin{array}{lll}
(Q^I[\gamma^k_J])^{2^d}&=&Q^{2^dI}[\gamma^k_J]^{2^d}\\
                       &=&Q^{2^dI}(Q^J[\gamma^k]*[-2^{l(J)}\gamma^k])^{2^d}\\
                       &=&Q^{2^dI}(Q^{2^dJ}[{2^d}\gamma^k]*[-2^{l(J)+d}\gamma^k]\\
                       &=&Q^{2^dI}(Q^{2^dJ}[0]*[-2^{l(J)}\gamma^k])\\
                       &=&0.\end{array}$$
Similarly, one can show that if $l(J)=d$, then $(Q^I[\gamma^k_J])^{2^d}=0$.\\
The observation that $[\gamma^k]$ is a $0$-dimensional class then implies that the Nishida relations are enough to describe the $A$-module structure of these subalgebras. This then completes the proof.
\end{proof}
Since, for different choices of the generators $\gamma^k$ there are no multiplicative, neither additive, relations among the classes of the form $Q^I[\gamma^k_J]$ we then can deduce that in general the homology of $H_*Q_0S^{-k}$ may contain tensor product of truncated polynomial algebras of different truncation heights.\\
Note that in the special case of the ${_2\pi_*J}$ we know the order of the generators, and hence we are able to determined the subalgebras more precisely. Finally, notice that $\sigma_*[\gamma^k]\neq 0$.

\begin{exm}
Consider $\eta\in{_2\pi_0}QS^{-1}\simeq\pi_1^S$ which is of order $2$, i.e. $Q^I[\eta]\in H_*Q_0S^{-1}$ for all nonempty $I$. Hence, classes of the form $Q^I[\eta]$ span an exterior algebra inside $H_*Q_0S^{-1}$. Applying the homotopy equivalence $*[\eta]:Q_0S^{-1}\to Q_{\eta}S^{-1}$ allows us to consider the exterior subalgebra of $H_*Q_{\eta}S^{-1}$ generated by $Q^I[\eta]*[\eta]$ . We then obtain a subalgebra of $H_*QS^{-1}$ given by polynomials in $[\eta]$ and $[0]$ with coefficients of the form $Q^I[\eta]$.\\
Note that in general, we don't know if all classes $Q^I[\gamma^k]$ are nontrivial. However, in this particular example, the classes $Q^I[\eta]$ with $I$ admissible are nontrivial. To see this notice that the mapping $\eta:S^0\to QS^{-1}$ extends to an infinite loop map $\eta:QS^0\to QS^{-1}$. On the level of $\pi_0$, or equivalently on the level of $H_0$, this induces the projection $\eta_*:\Z\to\Z/2$, i.e.
$$\eta_*[n]=\left\{  \begin{array}{ll}
                   [\eta] & \textrm{if $n$ is odd,}\\

                   [0]    & \textrm{if $n$ is even.}
                   \end{array}\right.
                   $$

Hence one may work out the image of $\eta_*:H_*QS^0\to H_*QS^{-1}$ as following
$$
\eta_*x_i  =  \eta_*(Q^i[1]*[-2])           =  \eta_*Q^i[1]*\eta_*[-2]           =  Q^i[\eta]*[0]           =  Q^i[\eta].          $$
Applying the homology suspension we obtain
$$\sigma_*Q^i[\eta]=\eta_*\sigma_*x_i=\eta_*Q^ig_1=Q^ix_1\neq 0$$
which shows that $Q^i[\eta]\neq 0$. Notice that $Q^ix_1$ depends on the Adem relation for $Q^iQ^1$ \cite[Lemma 5.10]{100} and will be trivial or maybe nontrivial modulo decomposable terms. However, the decomposable part always has nontrivial terms, for example $Q^3x_1=x_1^4$. In fact, as $x_1$ is primitive, $Q^ix_1$ is primitive. If $Q^iQ^1=0$, then the decomposable part of $Q^ix_1$ is the square of a primitive and one may work out what this primitive is nontrivial.\\
We recall from \cite[Theorem 5.34]{100} that as an $R$-module
$$H_*Q_0S^{-1}\simeq E_{\Z/2}(Q^Iw_{2i}:I\textrm{ admissible, }\ex(I)>2i,i\geqslant 0)$$
with $I\neq\phi$ if $i=0$, and $w_{2i}=(\Omega\lambda_\C)_*c_{2i}$ where $\lambda:Q\Sigma\C P_+\to Q_0S^0$ is the complex transfer, and $c_{2i}$ is the generator in $H_{2i}\C P_+$. Notice that $(I,2i)$ is not necessarily admissible. The classes $w_{2i}$ satisfy
$$\sigma_*w_{2i}=p_{2i+1}^{S^0}.$$
Moreover, two classes of the form $Q^Iw_{2i}$ maybe identified if they map to the same element under the homology suspension $\sigma_*:H_{*-1}Q_0S^{-1}\to H_*Q_0S^0$.\\
Now observe that instead of $\lambda_\C$ we could have used the iterated $J$-homomorphisms to define
$$w_{2i}=\left\{\begin{array}{l}
                (\Omega J_\R)_*c_{2i}\\
                (\Omega J_\C)_*c_{2i}
                \end{array}\right.$$
where we have used the same notation $c_{2i}$ to denote the corresponding generators in homology ring of $BU$ and $SO/U$. But observe that $BU$ as well as $SO/U$ both have bottom cells in dimension $2$, hence the cost of using these alternative definitions is that we do not see $w_0\in H_0Q_0S^{-1}$. But we can recover this by using the class $[\eta]$ where we have
$$H_*Q_0S^{-1}\simeq E_{\Z/2}(Q^Iw_{2i},Q^J[\eta]:I,J\textrm{ admissible, }\ex(I)>2i,\dim J>0)$$
with $J$ nonempty and $w_{2i}=(\Omega J_\R)_*c_{2i}$.
\end{exm}

We also have the following example from \cite[Note 5.41]{100}.

\begin{exm}
Consider $QS^{-9}$ with $\pi_0QS^{-9}\simeq\Z/2\oplus\Z/2\oplus\Z/2$. Let $\gamma_i$, $i=1,2,3$, be the generators for the first, second and third copies of $\Z/2$ in $\pi_0QS^{-9}$ respectively. The classes $Q^I[\gamma_i]$ live in the component $Q_{2\gamma_i}S^{-9}=Q_0S^{-9}$, for $I\neq\phi$. We then have exterior subalgebras inside $H_*Q_0S^{-9}$ generated by the symbols of the following forms,
$$\begin{array}{lllll}
Q^{I_1}[\gamma_1], & Q^{I_2}[\gamma_2], & Q^{I_3}[\gamma_3], & I_1,I_2,I_3\textrm{ admissible}.
\end{array}$$
As there are no relations among the generators $\gamma_i$ we see that there are no multiplicative relations among the three set of generators provided above. This might be seen as an evidence for some stable splitting of the space $Q_0S^{-9}$. We observe that within the homology algebra $H_*QS^{-9}$ terms of the form $Q^J[\gamma_i+\gamma_j]=Q^J([\gamma_i]*[\gamma_j])$ with $i\neq j$ can be calculated using the Cartan formula. We can apply similar techniques as we did before and use the translation maps $*[\gamma_i]:Q_0S^{-9}\to Q_{\gamma_i}S^{-9}$ and $*[\gamma_i+\gamma_j]:Q_0S^{-9}\to Q_{\gamma_i+\gamma_j}S^{-9}$ to obtain subalgebras of the other components, and hence we obtain a subalgebra of $H_*QS^{-9}$ which in this case will be an exterior algebra.
\end{exm}
Notice that in the both of the above examples, we had cyclic groups of order $2$ which was resulting in appearance of an exterior algebra in homology. If we consider the Hurewicz homomorphism
$$h:{_2\pi_0}QS^{-3}\to H_0QS^{-3}$$
then the generator of $\Z/8$ maps to a homology class say $[\nu]$ with $[\nu]^{*8}=[0]$. This then will implies that in the $\Z/2$-homology of $QS^{-3}$ there is a subalgebra generated by classes of the form $Q^I[\nu]$ each of which are of height $\leqslant8$ in the homology algebra, i.e.
$$(Q^I[\nu])^8=0.$$
We conclude by saying that it seems possible to generalise the results here to odd primes with required changes and probably with more complications.

\section{Proof of Theorem $6$}
We continue to work with $p=2$ and write $H_*$ for $H_*(-;\Z/2)$. The Eilenberg-Moore spectral sequence machinery \cite[Lemma 7.2, Proposition 7.3]{7} implies that if $X$ is a simply connected space and $H^*X$ is polynomial, then
$$H_*Q_0\Sigma^{-1}X\simeq E_{\Z/2}(\sigma_*PH_*QX)$$
and the homology suspension $\sigma_*:QH_*Q_0\Sigma^{-1}X\to PH_*QX$ is an isomorphism where $P$ and $Q$ are the primitive submodule and quotient indecomposable module functors.\\
According to \cite[Part I]{2} if $\{x_\alpha\}$ is an additive basis for $\overline{H}_*X$, the reduced homology of $X$ with $X$ path connected, then
$$H_*QX\simeq \Z/2[Q^Ix_\alpha:I\textrm{ admissible}, \ex(I)>\dim x_\alpha]$$
where a sequence of positive integers $I=(i_1,\ldots,i_r)$, called admissible, with $\ex(I)=i_1-(i_2+\cdots+i_r)$. We allow $I$ to be the empty sequence $\phi$ to be an admissible sequence, with $Q^\phi x=x$, and $\ex(\phi)=+\infty$.\\
According to \cite[Propositon 4.23]{10}, over $\Z/2$, a primitive class in $H_*QX$ is either a square of a primitive, or determines an indecomposable class belonging to the kernel of the square root map $QH_*QX\to QH_*QX$. The square root map $r:H_*X\to H_*X$ is given by
$rx=Sq^t_*x$ if $\dim x=2t$ and is trivial on the odd dimensional classes. The square root map $r:H_*QX\to H_*QX$ is described by
$$\begin{array}{lllll}
rQ^{2i+1}\xi= 0, & & rQ^{2i}\xi=Q^ir\xi,
\end{array}$$
where $\xi\in H_*QX$.\\
Note that any class $Q^Ix\in H_*QX$ with $\ex(Q^Ix)>0$, modulo decomposable terms, represents a unique indecomposable in $QH_*QX$. Since, $r$ is a multiplicative map, then $r:QH_*QX\to QH_*QX$ is determined by its action on the terms $Q^Ix$. We also note that in the cases of our consideration $r(Q^Ix+Q^Jy)=0$ only if $rQ^Ix=rQ^Jy=0$. Finally, for $Q^I$ with $i_j$ odd, we may write
$$Q^I=Q^JQ^{2k+1}+\textrm{ other terms },$$
where $J$ is an admissible term, and not necessarily $(J,2k+1)$, and other terms are determined by the `Adem relations' for the operations $Q^I$. It simply says that if $I$ has an odd entry at the middle then we can move it to the right end of the sequence.\\

Note that $X=BU,\Omega^k_0SO$, $k\equiv0,1,3,7$, have polynomial $\Z/2$-cohomology. Hence, we may apply the above machinery to calculate $H_*Q_0\Sigma^{-1}X$, where one passes to the universal cover of $Q_0\Sigma^{-1}X$ and then apply the machinery if necessary. We only need  to describe the submodule of primitives in $H_*QX$. We sketch the calculations for the case of $X=BU$ and leave the rest to the reader.\\
Note that $H_*BU\simeq\Z/2[c_{2i}:i>0]$ has an additive basis $\{c_J:J=(2j_1,\ldots,2j_t)\}$ with $c_J=c_{2j_1}c_{2j_2}\cdots c_{2j_t}$ and $0<2j_1\leqslant 2j_2\leqslant\cdots \leqslant 2j_t$. Hence,
$$H_*QBU\simeq\Z/2[Q^Ic_J:I\textrm{ admissible},\ex(I)>\dim J].$$
Let us write $\odot$ for the Pontrjagin product in $H_*QBU$ induced by the loop sum arising from applying $Q$ to $BU$. Notice that $c_J$ is not decomposable with respect to the $\odot$-product, e.g. $c_2\odot c_4$ is a $\odot$-decomposable whereas $c_{(2,4)}=c_2c_4$ is not a $\odot$-decomposable class.\\
For the square root map $r:QH_*QBU\to QH_*QBU$ we have $r(c_{2j_1}\cdots c_{2j_t})=(rc_{2j_1})\cdots(rc_{2j_t})$. For $J=(j_1,\ldots,j_t)$ we write $4|J$ if each entry of $J$ is a multiple of $4$, otherwise $4\not|J$. The above formula then implies that that $c_J\in \ker r$ if and only if $4\not| J$. Hence, $c_J$ with $4\not|J$ determines a unique primitive class in $H_*QBU$, say $p_J^{BU}$. For $4|J$ applying $Q^{2i+1}$ gives a term belonging to $\ker r$, hence we have a unique primitive class which modulo decomposable terms are defined by
$$p^{BU}_{i,J}=Q^{2i+1}c_J.$$
In order to make a distinction from primitive classes coming from $H_*BU$ we refer to primitive classes in $H_*QBU$ as $\odot$-primitive classes. We then have the following.

\begin{prp}
Any $\odot$-primitive class in $H_*QBU$ which is not a square can be written as a linear combination of classes of the form $Q^Ip_{i,J}^{BU}$ and $Q^Kp^{BU}_L$ with $4|J$ and $4\not|L$.
\end{prp}

\begin{proof}
Suppose $\xi\in H_*BU$ is a primitive class which is not a square. This implies that it has nontrivial image in $QH_*QBU$, and we may write
$$\xi=\sum Q^Ac_J+\odot-\textrm{decomposable terms}.$$
The fact that $\xi$ is $\odot$-primitive, implies that $\sum Q^Ac_J$ has to belong to the kernel of the square root map $r:QH_*QBU\to QH_*QBU$. An examination shows that if we have two terms of the form $Q^Ic_J$ and $Q^Kc_L$ with $r(Q^Ic_J)\neq0\neq r(Q^Kc_L)$ then $r(Q^Ic_J+Q^Kc_L)\neq 0$. Hence, each term $Q^Ac_J$ in the above expression for $\xi$ must belong to $\ker r$. This then means that either $A$ contains an odd entry, or $4\not|J$.
It is clear that if $4\not| J$ then modulo $\odot$-decomposable terms
$$c_J=p_J^{BU}.$$
We focus on the cases when $4|J$ and $A=(a_1,\ldots,a_t)$ has an odd entry. If $a_t=2i+1$ is odd, then modulo $\odot$-decomposable terms we have
$$Q^{a_t}c_J=p_{i,J}^{BU}.$$
Finally, if $a_t$ is even and the odd entry of $A$ is at the middle of the sequence we may use the Adem relations to move the odd entry to write and write
$$Q^A=Q^LQ^{2k+1}+\textrm{other terms}$$
where these ``other terms'' are of `excess less than $I$'. The proof then is complete by an induction on the excess.
This then implies that we may write $Q^I$ as sum given by
$$Q^A=\sum Q^LQ^{2k+1}$$
which means that
$$Q^Ac_J=\sum Q^LQ^{2k+1}c_J.$$
We then observe that modulo $\odot$-decomposable terms this later equality yields
$$Q^Ac_J=\sum Q^Lp_{k,J}.$$
Hence, we have shown that
$$\xi=\sum Q^Ip_{i,J}^{BU}+\sum Q^Kp_J^{BU}+D$$
where $D$ is a sum of $\odot$-decomposable terms. But notice $D$ is a decomposable primitive, so it must be square of a primitive. An induction on the length of the sequences $A$ in the initial expression for $\xi$ completes the proof.
\end{proof}

Next, define unique elements $c^{-1}_{i,J}\in QH_{2i+\dim J-1}Q\Sigma^{-1}BU$ and $c^{-1}_L\in QH_{\dim L-1}Q\Sigma^{-1}BU$ by
$$\begin{array}{lllll}
\sigma_* c^{-1}_{i,J}  =  p_{i,J}^{BU}, && \sigma_* c^{-1}_L  =  p_L^{BU}.
\end{array}$$
where $4|J$ and $4\not|L$. We then have
$$H_*Q\Sigma^{-1}BU\simeq E_{\Z/2}(Q^Ic^{-1}_{i,J},Q^Kc^{-1}_L:\ex(I)>2i+\dim L,\ex(K)>\dim L-1),$$
as an $R$-module, where we assume $I,K$ are admissible. Note that the classes $c^{-1}_{i,J}$ are not in the image of $(\Omega\iota_{BU})_*:H_*U\to H_*Q\Sigma^{-1}BU$, whereas $c^{-1}_L$s belong to the image of $(\Omega\iota_{BU})_*$ if and only if $L=(l_1,\ldots,l_t)$ is strictly increasing. Note that we require $L$ to be strictly increasing as then we can see that $u_L=u_{l_1}\cdots u_{l_t}\in H_*U$ maps to $c^{-1}_L$. But if we allow equal entries in $L$ we can not do this, as $H_*U$ is an exterior algebra. Finally, we note that two classes of the form $c^{-1}_{i,J}$ may be identified if they map to the same primitive class in $H_*QBU$ under the homology suspension $H_*Q\Sigma^{-1}BU\to H_{*+1}QBU$. Now we define define $w_{i,J}^{-2k-2}\in H_{2i+\dim J}Q_0S^{-2k-2}$ and $w^{-2k-2}_L\in H_{\dim L-1}Q_0S^{-2k-2}$ by
$$\begin{array}{llllll}
w_{i,J}^{-2k-2} =(\Omega j^\C_{2k+1})_*c^{-1}_{i,J}, && w^{-2k-2}_L =(\Omega j^\C_{2k+1})_*c^{-1}_L.
\end{array}$$
The next observation is the clear from the above theorem and the fact that $(\Omega j^\C_{2k+1})_*$ is a multiplicative map.
\begin{thm}
The homology algebra $H_*Q_0S^{-2k-2}$ contains an exterior algebra of the form
$$E_{\Z/2}(Q^Iw^{-2k-2}_{i,J},Q^Kw^{-2k-2}_L:\ex(I)>2i+\dim J,\ex(K)>\dim L-1)$$
where we assume $I,K$ are admissible. The classes $w^{-2k-2}_{i,J}$ are not in the image of $(\Omega^{2k+2}J_\C)_*$. Moreover, the classes $w^{-2k-2}_L$ are in the image of $(\Omega^{2k+2}J_\C)_*$ if and only if $L$ is strictly increasing. Notice that the classes belonging to $\im(\Omega^{2k+2}J_\C)_*$ are trivial if $2k\equiv 3,4$ mod $8$.
\end{thm}

\begin{proof}[(sketch)]
We only note that our claim on the image of $(\Omega^{2k+2}J_\C)_*$ are easy to verify once we recall that $\Omega^{2k+2}J_\C$ factors through $\Omega\iota_{BU}:U\to Q\Sigma^{-1}BU$. The fact that certain classes in $H_*Q\Sigma^{-1}BU$ are not in the image of $\Omega\iota_{BU}$ shows that the claimed classes do not belong to the image of $(\Omega^{2k+2}J_\C)_*$. We just observe that the numerical conditions at the end of the above theorem come from the fact that $\Omega^kJ_\C$ is trivial in $\Z/2$-homology if $k\equiv 5,6$ mod $8$.
\end{proof}

Now, let $k=0$ and choose $(I,2i+1,4j+1)$ to be an admissible sequence. We have $\sigma_*Q^Ic^{-1}_{i,(4j)}=Q^IQ^{2i+1}u_{2i+1}$ which maps to $Q^IQ^{2i+1}p^{S^0}_{2i+1}$ under $(J_\C)_*$, hence $Q^Iw^{-2}_{i,(4j)}\in H_*Q_0S^{-2}$ is nontrivial. This then proves the last part of our claim in Proposition 6. Notice that the above theorem defines classes that are not in the image of $J$-homomorphism, neither it is in the $R$-subalgebra generated by the image of $(\Omega^{2k+2}J_\C)_*$. This then proves Theorem 6, in the complex case. The case of $\Omega j^\R_k$ for $k\equiv0,1,3,7$ is similar.

\section{Comments on $H_*(\Omega^k_0\cok J;\Z/2)$}
At $p=2$, the homology algebra $H_*Q_0S^0$ is a polynomial algebra over generators $Q^Ix_i$. According to \cite[Theorem 3.1]{50} we may define $x_i=(J_\R)_*s_i$ where $s_i\in H_iSO$ is the generator, and $(J_\R)_*$ is an injection. This shows that in a loose sense, $\im(J_\R)_*$ and the Dyer-Lashof algebra $R$ are enough to generate $H_*Q_0S^0$; any $\xi\in H_*Q_0S^0$ can be represented as a polynomial in variables $x_i$ and coefficients from the Dyer-Lashof algebra $R$ (mod Adem relations). Notice that $H^*J\simeq H^*SO\otimes H^*BSO$, i.e. additively we know the homology of $J$, and all of this homology injects in $H_*Q_0S^0$. This then implies that as a space (at the prime $2$) almost any class of the form $Q^Ix_i$, expect those one mapping to $H_*BSO$, belong to $H_*\cok J$, i.e. we know about the additive homology of $\cok J$.\\
By `counting arguments' it is clear that in $H_*Q_0S^{-k}$ we may use similar technique at least to define infinitely many classes in $H_*\Omega^k_0\cok J$ which are not necessarily nontrivial, although we suspect that there are infinitely many of them are nontrivial. For $k>0$, the splitting (2.1) provides an equivalence of $k$-fold loop spaces which implies as algebras $H_*\cok J$ must be a subalgebra of $H_*Q_0S^{-k}$, unlike $k=0$. For instance, when $k=1$ we know that almost any class $Q^Iw_{2i}$ belongs to $H_*\Omega_0\cok J$ as well as we know that $H_*\Omega_0\cok J$ is an exterior algebra. A similar reasoning provides us with the following observation.
\begin{prp}
Almost any class of the form $Q^Ix^{-k}_i$ belongs to $H_*\Omega^k_0\cok J$.
\end{prp}
We note that in general we are unable to tell whether or not if $Q^Ix^{-k}_i\neq 0$, however, if $k\leqslant 7$ then the definition of the classes $x^{-k}_i$ together with Theorem $2$ tells us that some of these classes do survive to nontrivial classes in $H_*Q_0S^0$. In these cases, we then obtain infinitely many classes in $H_*\Omega^k_0\cok J$. Notice that this does not tell us about the $E_\infty$-algebra structure of $H_*\Omega^k_0\cok J$, but it tells us how operations coming from the $k$-fold loop structure should behave.

\section{Equivariant $J$-homomorphisms}
The effect of the $\Z/2$- and $S^1$-equivariant $J$-homomorphisms
$$\begin{array}{lllllllll}
J_{\Z/2}:SO\longrightarrow Q_0 P_+, & & J_{S^1}:U\longrightarrow  Q\Sigma\C P_+,
\end{array}$$
in $\Z/p$-homology is known due to Mann-Miller-Miller \cite[Corollary 5.3]{9}. Analogous to the non-equivariant case, we may consider
$$\begin{array}{lllllllll}
\Omega^kJ_{\Z/2}:\Omega^k_0SO\longrightarrow \Omega^k_0Q_0 P_+, & & \Omega^k_0J_{S^1}:\Omega^k_0U\longrightarrow \Omega^k_0Q\Sigma\C P_+.
\end{array}$$
Similar to Theorem 2, using the Bott periodicity maps, we may define a set of classes in the $\Z/p$-homology of $\Omega^k_0Q_0 P_+$ and $\Omega^k_0Q\Sigma\C P_+$ in a geometric way, which allows to determine their action of the Steenrod operations $Sq^r_*$ on these classes.\\
The inclusion $\Z/p^n\to S^1$ yields equivariant transfer maps
$$\begin{array}{lllllllll}
t_2    :\Sigma\C P_+\longrightarrow Q_0 P_+ & & t_{p,n}:\Sigma\C P_+\longrightarrow  QB\Z/p^n_+
\end{array}$$
extending to infinite loop maps
$$\begin{array}{lllllllll}
t_2    :Q\Sigma\C P_+\longrightarrow Q_0 P_+ & & t_{p,n}:Q\Sigma\C P_+\longrightarrow  QB\Z/p^n_+.
\end{array}$$
We refer the reader to \cite{9} for homology of these maps. Similarly, the iterated loops of these mappings, composed with the iterated equivariant $J$-homomorphisms above, can be used to define classes in the $\Z/p$-homology algebras of $\Omega^kQ_0B\Z/p^n_+$ with $p\geqslant 2$. For example, if $p>2$ the mapping $t_{p,1}:\Sigma\C P_+\to QB\Z/p^n_+$ is an injection in homology. The effect of $t_{p,1}:Q\Sigma\C P_+\longrightarrow QB\Z/p_+$ in homology on the generators of $H_*(\Sigma\C P_+;\Z/p)$ is the same as the effect of $t_{p,1}:\Sigma\C P_+\to QB\Z/p_+$ in homology. On the other hand observe that the generators of $H_*(\Sigma\C P_+;\Z/p)$ are in the image of $(J_{S^1})_*$. These together prove the following fact.
\begin{lmm}
The mapping
$$U\longrightarrow Q\Sigma\C P_+\longrightarrow QB\Z/p_+$$
induces an injection in $\Z/p$-homology on the generators $u_{2i+1}\in H_{2i+1}(U;\Z/p)$, when $p>2$.
\end{lmm}
We may also use homology suspension arguments to see that the $\Z/p$-homology of $\Omega^kJ_{S^1}$ define infinitely many classes in homology of $\Omega^k_0QB\Z/p_+$ at least for $k\leqslant 3$.\\
We note that in the context of equivariant homotopy theory, there seem to be more complications. For instance, $\pi_0^{S^1}U\not\simeq 0$ where $\pi_*^{S^1}$ is the $S^1$-equivariant homotopy. This means that while restricting to the base point in the non-equivariant case, we had $U\to SG$, where in the equivariant case, the `equivariant base point component' of $U$ is not quite that straightforward. If one takes the component correspoding to the identity, then we get the ring of real representations. Similar complications arise, when we consider possibility of having $S^1$-equivariant versions of $J$, and Sullivan's splitting and so on. For example, $J$ at $p=2$ may be defined as the fibre of $\psi^3-1:\Omega^\infty_0 ko\to\Omega^\infty_0 ko$ where $ko$ denotes the connective $KO$-theory. However, the spaces representing this equivariant theory are more complicated, and are not like the spaces in the real or complex Bott periodicity, even in the case of $ku$ \cite{51}. We leave further investigation on this to a future work.\\

I am very grateful to John Greenlees for comments on the equivariant cases which made me to realise that the problem in the context of equivariant homotopy theory is more complicated, and more interesting, than it seemed to me at first place.


\begin{thebibliography}{MMMMMMMMM}

\bibitem[A66] {0} J. F. Adams `On the groups $J(X)$ IV', \textit{Topology} 5 (1966) 21--71


\bibitem[A78] {1} John Frank Adams `Infinite loop spaces' Annals of Mathematics Studies, 90. \textit{Princeton University Press, Princeton, N.J.; University of Tokyo Press, Tokyo,} (1978)

\bibitem[AE00] {42} Mohammad A. Asadi-Golmankhaneh \and Peter J. Eccles `Double point self-intersection surfaces of immersions' \textit{Geom. Topol.} 4 (2000), 149--170


\bibitem[BS74] {21} J. C. Becker \and R. E. Schultz `Equivariant function spaces and stable homotopy theory. I' \textit{Comment. Math. Helv.} 49 (1974), 1--34

\bibitem[BS67]{41} J. M. Boardman \and B. Steer `On Hopf invariants', \textit{Comment. Math. Helv.} 42 (1967) 180--221


\bibitem[CLM]{2} F. R. Cohen, T. J. Lada, \and J. P. May `The homology of iterated loop spaces', Lecture Notes in Mathematics, Vol. 533. \textit{Springer-Verlag, Berlin-New York,} 1976. 


\bibitem[C60]{3} H. Cartan  `D\'{e}monstration homologique des th\'{e}or\`{e}mes de p\'{e}riodicit\'{e} de Bott, II. Homologie et cohomologie des groupes classiques et de leurs espaces homog\`{e}nes' \textit{Séminaire Henri Cartan}, 12 no. 2, 1959-1960
                       Périodicité des groupes d'Homotopie stables des groupes classiques, d'après Bott

\bibitem[CP89]{52} F. R. Cohen \and F. P. Peterson `Some remarks on the space ${\rm Im}\,J$' \textit{Algebraic topology (Arcata, CA, 1986), 117--125, Lecture Notes in Math., 1370, Springer, Berlin}, 1989


\bibitem[CP89]{4} F. R. Cohen \and F. R. Peterson `On the homology of certain spaces looped beyond their connectivity', \textit{Israel J. Math.} 66 (1989), no. 1-3, 105--131

\bibitem[C75]{5} Edward B. Curtis `The Dyer-Lashof algebra and the $\Lambda $-algebra', \textit{Illinois J. Math.} 19 (1975), 231--246


\bibitem[D60] {6} A. Douady `P\'{e}riodicit\'{e} du groupe unitaire', \textit{Séminaire Henri Cartan}, 12 no. 2, 1959-1960 Périodicité des groupes d'Homotopie stables des groupes classiques, d'après Bott

\bibitem[E80]{40} Peter John Eccles `Multiple points of codimension one immersions of oriented manifolds' \textit{Math. Proc. Cambridge Philos. Soc.} 87 (1980), no. 2, 213--220


\bibitem[G04]{7} S{\o}ren Galatius `Mod $p$ homology of the stable mapping class group'  \textit{Topology} 43 (2004), no. 5, 1105--1132

\bibitem[G04]{51} J. P. C. Greenlees `Equivariant connective $K$-theory for compact Lie groups' \textit{J. Pure Appl. Algebra} 187 (2004), no. 1-3, 129--152

\bibitem[HW83]{20} Henning Hauschild \and Stefan Waner `The equivariant Dold theorem mod $k$ and the Adams conjecture'
                      \textit{Illinois J. Math.} 27 (1983), no. 1, 52--66.

\bibitem[KP72]{50} Daniel S. Kahn \and Stewart B. Priddy `Applications of the transfer to stable homotopy theory' \textit{Bull. Amer. Math. Soc.} 78 (1972), 981--987


\bibitem[K71]{18} Stanley O. Kochman `The homology of the classical groups over the Dyer-Lashof algebra'
                      \textit{Bull. Amer. Math. Soc.} 77 (1971), 142--147


\bibitem[M70]{14} Ib Madsen `On the action of the Dyler-Lashof algebra in $H_*(G)$ and $H_*(G/TOP)$' PhD thesis, \textit{The University of Chicago}, 1970


\bibitem[M75]{8} Ib Madsen `On the action of the Dyer-Lashof algebra in $H_*(G)$.'  \textit{Pacific J. Math.} 60 (1975), no. 1, 235--275.


\bibitem[MMM86]{9} Benjamin M. Mann, Edward Y. Miller, \and Haynes R. Miller `$S^1$-equivariant function spaces and characteristic classes', \textit{Trans. Amer. Math. Soc.} 295 (1986), no. 1, 233--256

\bibitem[McC01]{17} John McCleary `A user's guide to spectral sequences' Second edition.  Cambridge Studies in Advanced Mathematics, 58. \textit{Cambridge University Press, Cambridge}, (2001)

\bibitem[MM65]{10} J. W. Milnor \and J. C. Moore `On the structure of Hopf algebras', \textit{Ann. of Math.} (2) 81 1965 211--264.


\bibitem[P75]{11} Stewart Priddy `Dyer-Lashof operations for the classifying spaces of certain matrix groups', \textit{Quart. J. Math. Oxford Ser.} (2) 26 (1975), no. 102, 179--193

\bibitem[Q71]{12} Daniel Quillen `The Adams conjecture', \textit{Topology} 10 (1971) 67--80

\bibitem[R86]{26} Douglas C. Ravenel `Complex cobordism and stable homotopy groups of spheres'. \textit{Pure and Applied Mathematics, 121. Academic Press, Inc., Orlando, FL,} (1986)

\bibitem[S76]{16} Victor Snaith `Infinite loop maps and the complex J-homomorphism' Bull. Amer. Math. Soc. 82 (1976), no. 3, 508--510.

\bibitem[S79]{13} Victor P. Snaith `Algebraic cobordism and $K$-theory' \textit{Mem. Amer. Math. Soc.} 21 (1979), no. 221

\bibitem[W78]{99} George W. Whitehead `Elements of homotopy theory' Graduate Texts in Mathematics, 61. \textit{Springer-Verlag}, New York-Berlin, 1978

\bibitem[Z09]{100} Hadi Zare `On Spherical Classes in $H_*QS^n$' PhD Thesis, \textit{The University of Manchester}, (2009)
                       available at \textrm{http://eprints.ma.man.ac.uk/1290/01/covered/MIMS{\_}ep2009{\_}44.pdf}

\end{thebibliography}
\end{document}